\documentclass{amsart}

\usepackage[a4paper, total={6in, 8in}]{geometry}

\usepackage{amsthm}
\usepackage{amsmath}
\usepackage{amssymb}
\usepackage{comment}
\usepackage{xifthen}

\usepackage{tikz-cd}

\usepackage{graphicx}

\usepackage[dvipsnames]{xcolor}

\usepackage{hyperref}
\usepackage{doi}

\theoremstyle{plain}
\newtheorem{theorem}{Theorem}[section]
\newtheorem{observation}[theorem]{Observation}
\newtheorem{corollary}[theorem]{Corollary}
\newtheorem{lemma}[theorem]{Lemma}
\newtheorem{proposition}[theorem]{Proposition}

\newtheorem{question}[theorem]{Question}

\theoremstyle{definition}
\newtheorem{definition}[theorem]{Definition}
\newtheorem{example}[theorem]{Example}
\newtheorem{remark}[theorem]{Remark}

\newcommand{\A}{\mathcal{A}}
\newcommand{\B}{\mathcal{B}}
\newcommand{\T}{\mathcal{T}}
\newcommand{\C}{\mathcal{C}}

\newcommand{\Z}{\mathbb{Z}}
 
\newcommand{\nat}{\omega} 

\newcommand{\WP}{\operatorname{WP}}
\newcommand{\IP}{\operatorname{IP}}
\newcommand{\lex}{\operatorname{lex}}

\newcommand{\eqce}{=^{ce}}
\newcommand{\emin}{E_{\min}}
\newcommand{\emax}{E_{\max}}
\newcommand{\ezero}{E_0^{ce}}

\newcommand{\fg}{f.g.}

\newcommand{\AG}[1]{{\mathcal{AG}_{#1}}}
\newcommand{\CM}[1]{{\mathcal{CM}_{#1}}}
\newcommand{\CS}[1]{{\mathcal{CS}_{#1}}}
\newcommand{\UF}[2]{{\mathcal{UF}^{#1}_{#2}}}
\newcommand{\V}[1]{{\mathcal{V}_{#1}}}

\newcommand{\iAG}[1]{\cong_{\AG{#1}}}
\newcommand{\iCM}[1]{\cong_{\CM{#1}}}
\newcommand{\iCS}[1]{\cong_{\CS{#1}}}
\newcommand{\iUF}[2]{\cong_{\UF{#1}{#2}}}
\newcommand{\iV}[1]{\cong_{\V{#1}}}

\title{Isomorphism relations on classes of c.e. algebras}

\author[Meng-Che ``Turbo'' Ho]{Meng-Che ``Turbo'' Ho}
\address{Natural Sciences Division, New College of Florida}
\email{\href{mailto:mho@ncf.edu}{mho@ncf.edu}}

\author[Martin Ritter]{Martin Ritter}
\address{Institute of Discrete Mathematics and Geometry, Vienna University of Technology}
\email{\href{mailto:martin.ritter@tuwien.ac.at}{martin.ritter@tuwien.ac.at}}

\author[Luca San Mauro]{Luca San Mauro}
\address{Department of Humanistic Research and Innovation, University of Bari}
\email{\href{mailto:lucafrancesco.sanmauro@uniba.it}{lucafrancesco.sanmauro@uniba.it}}



\smallskip

\thanks{The first author is partially supported by the National Science Foundation under Grant No.~DMS-2054558 and No.~DMS-2452314. The second author would like to thank Steffen Lempp for many helpful discussions during the initial stages of this work. The third author is a member of INDAM-GNSAGA. The authors would like to thank Uri Andrews for a useful discussion that led to Theorem \ref{iCM-non-reduction}.}

\begin{document}

    \begin{abstract}
We investigate the complexity of isomorphism relations for classes of \emph{finitely generated} and $n$-generated computably enumerable (c.e.) algebras, presented via c.e.\ presentations---that is, as quotients of term algebras over decidable sets of generators by c.e.\ congruences. Our goal is to develop a systematic framework for analyzing such isomorphism problems from a computability-theoretic perspective. To compare their complexity, we employ the notion of \emph{computable reducibility}, measuring these relations against canonical benchmarks on c.e.\ sets, such as $\eqce$, $\ezero$, and the ordinal-indexed family $\emin(\alpha)$. A central insight of our work is the interplay between the algebraic structure and the algorithmic complexity: we show that if every algebra in a class satisfies the ascending chain condition on its congruence lattice, then the corresponding isomorphism relation is computably reducible to $\eqce$. We also apply this framework to a range of concrete cases. In particular, we analyze the isomorphism relations for finitely generated commutative semigroups, monoids, and groups, positioning them within the broader landscape of classification problems.
\end{abstract}

	\maketitle

\section{Introduction}

The complexity of isomorphism problems occupies a central role in mathematical logic, connecting areas as diverse as theoretical computer science, combinatorics, and descriptive set theory. In theoretical computer science, the classic graph isomorphism problem famously stands as one of the few natural problems neither known to be tractable nor proven NP-complete (see, e.g.,~\cite{kobler2012graph}). In descriptive set theory, the classification of isomorphism relations for natural classes of countable structures---pioneered by Friedman and Stanley \cite{friedman1989borel}---relies on the framework of Borel reducibility, which provides a systematic means of comparing the complexity of equivalence relations and, thereby, the feasibility of classifying mathematical objects up to isomorphism by ``nice'' invariants.

Within this framework, it is now well understood that, for many familiar classes of structures $\mathfrak{K}$ (such as graphs, linear orders, fields, nilpotent groups~\cite{friedman1989borel}; Boolean algebras~\cite{camerlo2001completeness}; torsion-free abelian groups~\cite{paolini2024torsion}), the isomorphism relation attains maximal complexity: for any countable language $L$, the relation $\cong_L$ Borel reduces to $\cong_{\mathfrak{K}}$. This result places strong constraints on the prospects of classification: any attempt to parametrize isomorphism types by ``explicit'' invariants is necessarily as complex as possible, at least from the perspective of Borel reducibility.

Computable structure theory further refines the notion of complexity by adding explicit effectiveness constraints. One may pursue this by considering \emph{Turing computable embeddings}~\cite{knight2007turing}, which require the reduction between isomorphism relations to be witnessed by Turing operators. Alternatively, and more relevant to our context, one can focus on the complexity of isomorphism relations among \emph{computable structures}~\cite{fokina2009equivalence,fokina2012isomorphism,montalban2016classes}, that is, structures with domain $\omega$ (the set of natural numbers) whose atomic diagram is computable. Each such structure is coded by a natural number (the index of a computable presentation), so the isomorphism relation for computable members of a class $\mathfrak{K}$ becomes an equivalence relation on a certain index set $I(\mathfrak{K}) \subseteq \omega$,  where $\omega$ denotes the set of natural numbers.

  Here, the central notion is that of \emph{computable reducibility}, where equivalence relations on indices of computable structures are compared via partial computable functions (see Section \ref{sec: comp red} for formal definitions). Computable reducibility thus serves as a computable analogue of Borel reducibility---an analogy explored extensively in the literature (see, e.g., \cite{gao2001computably, coskey2012hierarchy,clemens2023computable,andrews2024investigating,andrews2025analogues})---and naturally leads to intricate degree structures whose subtle properties have been extensively investigated (see, e.g., \cite{andrews2020theory,andrews2019joins,andrews2023structure}).
 Intriguingly, recent research reveals significant qualitative differences between the complexity of the isomorphism relations measured in the computable setting versus the descriptive set-theoretic context, with notable examples arising from, e.g., $p$-groups (compare \cite[Theorem 5]{friedman1989borel} and \cite[Theorem 4]{fokina2012isomorphism}).

In this paper, we initiate a systematic investigation into the complexity of the isomorphism problem for classes of \emph{computably enumerable (c.e.) algebras} arising from natural algebraic varieties. Our approach is inspired by classical questions from universal algebra regarding the algorithmic complexity of algebraic presentations, now viewed through an explicitly computational lens. Traditionally, the isomorphism problem for a variety of algebraic structures asks whether two algebras specified by finite presentations —-- consisting of finitely many generators and defining relations --— are isomorphic (see~\cite{kharlampovich1995algorithmic} for an overview). A natural extension of this classical scenario is to consider \emph{c.e.\ presentations}, where an algebra is given as the quotient of a term algebra over a computable set of generators by a c.e.\ congruence relation (see Section~\ref{sec:ce_pres}). Early foundational work on c.e.\ algebras was developed by Khoussainov and Selivanov~\cite{khoussainov2018journey, selivanov2003positive}, with subsequent investigations appearing in~\cite{gavryushkin2016reducibilities, gavruskin2014graphs, fokina2016linear, khoussainov2014finitely}.

A salient advantage of the c.e.\ framework is that, once a variety $\mathcal{V}$ is fixed, every natural number can be interpreted as an index for a c.e.\ presentation of a member of $\mathcal{V}$—thus, the relevant domain for the isomorphism relation is all of $\omega$. This allows us to consider isomorphism as a total equivalence relation on the natural numbers, and to investigate its complexity under \emph{total} computable reductions. Also, the c.e.\ framework is faithful to the origin of the isomorphism problem as proposed by Dehn.

Among the fundamental decision problems in combinatorial algebra, the isomorphism problem stands out as one of the three core problems introduced by Dehn~\cite{dehn1911unendliche} at the origins of the subject. While recent research has examined the complexity of the word problem for c.e.\ algebras---see~\cite{delle2020word, delle2023classifying, andrews2024algorithmically}, and references therein, in particular \cite{kasymov1986finitely}---relatively little is known about the complexity of the isomorphism problem in this setting. Some notable exceptions are the result of Miller \cite{miller1971group} that shows the isomorphism relation for finitely presented groups is $\Sigma^0_1$-complete under computable reducibility (although he did not use modern terminologies as they are not defined yet), and the result of Andrews, Harrison-Trainor, and Ho~\cite{andrews2024two} that shows the isomorphism relation for $6$-generated groups is $\Sigma^0_3$-complete under computable reducibility.

In this paper, we develop the first systematic account of the complexity of isomorphism relations for classes of 
$n$-generated and finitely generated c.e.\ algebras, thereby contributing to the understanding of the complexity landscape for finitely generated algebras (see, e.g.,\cite{thomas1999complexity} for the descriptive set-theoretic perspective, and \cite{harrison2017finitely} for the computable perspective).

To calibrate the complexity of these isomorphism relations, we employ the framework of computable reducibility and consider a collection of natural benchmark equivalence relations, following Coskey, Hamkins, and Miller~\cite{coskey2012hierarchy}. For our classification, two relations play a pivotal role:
\begin{itemize}
\item $\eqce$, the $\Pi^0_2$ equality relation on c.e.\ sets;
\item $\ezero$, the $\Sigma^0_3$ relation of eventual equality on c.e.\ sets, which is, in fact, the maximal complexity attainable for isomorphism relations of finitely generated c.e.\ algebras.
\end{itemize}

A central insight of our work is that natural algebraic properties of a variety impose upper bounds on the complexity of the corresponding isomorphism relation. In particular, we establish the following:
\begin{itemize}
\item (informal version of Theorem~\ref{acc_upper_bound}) \emph{If a class of c.e.\ algebras has the ascending chain condition (ACC) on its congruence lattices, then the associated isomorphism relation is computably reducible to~$\eqce$}. 
\end{itemize}

Through our analysis, we systematically assess the complexity of the isomorphism relation for a broad spectrum of natural classes of c.e.\ algebras, including finitely generated commutative semigroups, monoids, and groups (see Figure \ref{fig:below=ce}). To further refine our classification tools, we extend classical benchmark relations (such as $\emin$) using ordinals, yielding new benchmarks to elucidate the fine-grained complexity distinctions within these hierarchies.

    \subsection{Organization of this paper}	
In Section \ref{sec: background}, we provide the necessary background on universal algebra, c.e.\ presentations, isomorphism relations, computable reducibility, and standard benchmark equivalence relations on c.e.\ sets. Section \ref{sec: isorelations and emin} introduces a range of natural examples whose isomorphism relations fall into the degrees $\emin(\omega \cdot n)$ for $n \geq 1$. In Section \ref{sec:below_eqce}, we undertake a more systematic investigation of isomorphism relations below $\eqce$, proving several general results---including Theorem \ref{acc_upper_bound}---and extending the classification of natural examples initiated in Section \ref{sec: isorelations and emin}. Section \ref{sec:above_eqce} turns to isomorphism relations above $\eqce$, where we present a new example of a $\Sigma^0_3$-complete isomorphism relation. Finally, Section \ref{sec:questions} outlines a collection of open questions and proposes directions for future research.

\section{Basic background}\label{sec: background}
	\subsection{Computability Theory}
    We assume familiarity with the basic notions from computability theory (found, e.g.,\ in \cite{odifreddi1992classical}). Our notation is standard; most importantly, $\nat$ is the set of natural numbers, 
    $\varphi_e$ is the $e$-th partial computable function, $W_e = \{x \in \nat\mid  \varphi_e(x)\downarrow\}$ is the $e$-th c.e.\ set, and $W_{e,s}$ is the restriction of $W_e$ to elements enumerated before stage $s$. We denote the Turing jump of a set $X$ by $X^\prime$.

	\subsection{Universal algebra}
	We very briefly introduce some notions from universal algebra that we need. A more thorough description can be found, e.g., in \cite{gratzer2008universal}.
    
		An \emph{algebraic signature} is a family of function symbols $(f_i)_{i \in I}$, together with a family of natural numbers $(n_i)_{i \in I}$, called \emph{arities}, for some index set $I$. An \emph{algebra} over such a signature is given by a tuple $\mathcal{A} = (A,(f_i^\mathcal{A})_{i \in I})$, where $A$ is a set and $(f_i^\mathcal{A})_{i \in I}$ a family of functions $f_i^\mathcal{A}: A^{n_i} \to A$. We call $f_i$ a \emph{constant} if $n_i = 0$.

    We will focus only on algebras where the index set $I$ is decidable and the arity function that maps $i \in I$ to $n_i$ is computable. Since we will only deal with algebraic signatures, we will just call them signatures.
	
		Let $\mathcal{A} = (A,(f_i^\mathcal{A})_{i \in I})$ be an algebra. An equivalence relation $E \subseteq A^2$ is called a \emph{congruence relation} if for every function symbol $f_i$ with arity $n_i$
		\[
		\forall (a_1,b_1),...,(a_n,b_n) \in E: f^\mathcal{A}_i(a_1,...,a_n) ~E~ f^\mathcal{A}_i(b_1,...,b_n)
		\]
		holds. We then call $\mathcal{A}/E := (\{[a]_E \mid  a \in A\}, (f_i^{\mathcal{A}/E})_{i \in I})$ a \emph{quotient} of $A$, where $[a]_E$ is the equivalence class of $a$ under $E$ and $f_i^{\mathcal{A}/E}([a_1]_E,...,[a_{n_i}]_E) := [f_i^\mathcal{A}(a_1,...,a_{n_i})]_E$.

	Let $\mathcal{A} = (A,(f_i^\mathcal{A})_{i \in I}), \mathcal{B} = (B,(f_i^\mathcal{B})_{i \in I})$ be two algebras over the same signature. A function $\varphi: A \to B$ is a \emph{homomorphism} if for every function symbol $f_i$ with arity $n_i$
	\[
	\forall a_{1},...,a_{n_i} \in A: \varphi(f_i^\mathcal{A}(a_1,...,a_{n_i})) = f_i^\mathcal{B}(\varphi(a_1),...,\varphi(a_{n_i}))
	\]
	holds. If $\varphi$ is surjective, we say $\B$ is a \emph{homomorphic image} of $\A$, denoted $\A \twoheadrightarrow \B$. If $\varphi$ is bijective, we say $\varphi$ is an \emph{isomorphism} and $\mathcal{A}$ is \emph{isomorphic} to $\mathcal{B}$, denoted $\mathcal{A} \cong \mathcal{B}$. 
    Given a homomorphism $\varphi: A \to B$, let $\ker(\varphi) := \{(a_1,a_2) \in A \times A \mid \varphi(a_1) = \varphi(a_2)\}$ be the \emph{kernel} of $\varphi$. It is easy to check that $\ker(\varphi)$ is a congruence relation on $A$. By the well-known first isomorphism theorem (or by the fundamental theorem on homomorphisms), if $A \twoheadrightarrow \B$, witnessed by $\varphi$, then $\B$ is isomorphic to the quotient of $\A$ by $\ker(\varphi)$. Also, it is well known that if $\B$ is a quotient of $\A$ then $\A \twoheadrightarrow \B$.

Below, we state the third isomorphism theorem, which we need when working with presentations of algebras. A proof of it can be found in most algebra textbooks (e.g.\ \cite{gratzer2008universal}, where it is called the second isomorphism theorem).
\begin{theorem}[Third Isomorphism Theorem]\label{third_isom_theorem}
	Let $\A$ be an algebra and $C_1$, $C_2$ two congruence relations on $\A$ with $C_1 \subseteq C_2$. Then $D := \{([a]_{C_1},[b]_{C_1}) \mid  a,b \in C_2\}$ is a congruence relation on $\A/C_1$ and $(\A/C_1)/D \cong \A/C_2$. Therefore, $\A/C_1 \twoheadrightarrow \A/C_2$.
\end{theorem}

	For a given signature $\sigma = (f_i)_{i \in I}$, the \emph{term algebra} $\T_\sigma(X)$ over the set of variables $X$ has as its domain the set of all $\sigma$-terms over $X$, denoted $T_\sigma(X)$, and functions are defined as follows:
    \[
    f^{\T_\sigma(X)}_i: T_\sigma(X)^{n_i} \to T_\sigma(X), (t_1,..., t_{n_i}) \mapsto f_i(t_1,...,t_{n_i}).
    \]
    If $\sigma$ is clear from the context, we will often omit it. 

		An \emph{equational law} is a pair of terms $(s,t)$, also written $s \approx t$, over some signature $\sigma$ and with free variables. 	
		Let $\sigma$ be a signature and let $\Lambda$ be a set of equational laws over $\sigma$. Then the class
		\[
		\mathcal{V} = \{\mathcal{A}\mid   \mathcal{A} \text{~is an algebra over $\sigma$ and ~} \forall (s,t) \in \Lambda: \mathcal{A} \vDash s = t\}
		\]
		is called an \emph{variety} or \emph{equational class}. By Birkhoff's HSP theorem, this is equivalent to saying that $\mathcal V$ is closed under taking homomorphic images, subalgebras, and products.

	\subsection{C.e.\ presentations}\label{sec:ce_pres}
    In universal algebra, a \emph{presentation} (in a fixed signature) is a generating set $X$ and a set $R$ of equations of terms over the generators, written $\langle X, R \rangle$. Let $\approx_R$ be the congruence relation \emph{generated by} $R$, i.e.\ $\approx_R$ is the smallest congruence relation containing $R$ (or equivalently, the binary relation containing $R$ and all equations that can be derived from $R$ through symmetry, transitivity, and function application). An algebra is \emph{presented by} $\langle X, R \rangle$ if it is isomorphic to $\T(X)/\approx_R$. To define presentations in a fixed variety (e.g.,\ group presentations), one can have $\approx_R$ be generated by $R$ together with the equational laws of the variety. In this paper, since we will deal with different varieties, we will write $\langle X, R\rangle_\V{}$ for a presentation in a variety $\V{}$.

\smallskip

    The following definition can be found in e.g.\ \cite{delle2020word}:
    \begin{definition}[\cite{delle2020word}]\label{ce_pres_in_varieties}
			A presentation $\langle X,R \rangle_\V{}$ is a \emph{c.e.\ presentation} if 
            \begin{itemize}
                \item $\mathcal{V}$ has a c.e.\ set of equational laws,
                \item $X$ is decidable, and
                \item $R$ is c.e.
            \end{itemize}  
		\end{definition}

        This notion agrees with the common definition of c.e.\ algebras (by e.g.\ \cite{khoussainov2018journey}) in that c.e.\ algebras are exactly the algebras that have a c.e.\ presentation. So to keep things simple, for this paper, we define a \emph{c.e.\ algebra} as an algebra that has a c.e.\ presentation.

    For our study of isomorphism relations on classes of c.e.\ algebras, given by c.e.\ presentation, we need to be able to encode presentations as natural numbers.

    First, observe that, with a signature and a generating set given, one can fix a bijective encoding of pairs of terms over $X$ into natural numbers. Therefore, we can think of the set of identities in a presentation as a set of natural numbers encoding identities and just write $W_e$ for the set of identities encoded by $W_e$. So within a variety, we can identify a presentation by the natural number $e$ \emph{if} the generating set is fixed.

    As the set of generators must by definition be decidable, it can (up to renaming) only be one of $X_\nat:= \{x_n \mid n \in \nat\}$ and $X_n:= \{x_0,...,x_{n-1}\}$ for $n \in \nat$. So we can immediately define an indexing of all c.e.\ presentation for the class of c.e.\ algebras in a variety and for the class of $n$-generated c.e.\ algebras in a variety (note that requiring the algebras to be presented using exactly $n$ generators is not a restriction since a presentation with a smaller generating set can always be simulated by collapsing the superfluous generators).

    \begin{remark}
If the set of pairs of terms is finite (this can happen if there are no function symbols of arity $\geq 1$) one can fix an encoding into an initial segment of the natural numbers. Then, for a given c.e.\ set $W_e$, some elements in $W_e$ may not encode any identity. In this paper, such a variety only occurs in Proposition \ref{least_isomorphism_relation}.
	\end{remark}
	\begin{definition}
		 Let $\V{}$ be a variety. Then we define \[\mathcal{V}[e]:= \langle X_\nat \mid  W_e\rangle_\mathcal{V}\] and for any $n \geq 1$ \[\mathcal{V}_n[e]:= \langle X_n \mid  W_e\rangle_\mathcal{V}.\]

	\end{definition}

	For the class of finitely generated algebras of a variety, we cannot fix the set of generators; however, we can still encode presentations in a natural way.
	
	\begin{definition}
    Let $\V{}$ be a variety. Then we define
		\[
		\mathcal{V}_{f.g.}[\langle k, e\rangle] = \langle X_{k} \mid  W_e \rangle_\mathcal{V}
		\]
	\end{definition}
    We will also use $\V{n}$ to refer to the class of $n$-generated c.e.\ algebras from $\V{}$ and similarly $\V{\fg}$ for the class of finitely generated c.e.\ algebras from $\V{}$. In other words, $\V{n} = \{\V{n}[e] \mid e \in \omega\}$ and $\V{\fg} = \{\V{\fg}[\langle k,e\rangle] \mid k,e \in \omega\}$.
    
    Note that in the above definition, we excluded $0$-generated algebras, since two presentations from such a class can only be isomorphic through the identity mapping. That is because the term algebra is generated by the constants and isomorphisms map a constant to itself. Furthermore, when we talk about finitely generated structures, i.e.\ $\V{n}$ and $\V{\fg}$ for some variety $\V{}$, we will only consider varieties over a finite type. Otherwise, we could have types with infinitely many constants, which defeats the purpose of restricting to finite generation.

    Here are some simple but important examples for presentations in a given class, namely the empty and the full set of identities.
    \begin{example}
        Let $e_\emptyset$ be an index for the empty set. Then $\V{n}[e_\emptyset] = T(X_n)/\approx_\emptyset$, where $\approx_\emptyset$ contains only the identities of the variety. That is, $\V{n}[e_\emptyset]$ is the free algebra in $\V{}$ over $n$ variables. Also, notice that if we take instead of $\emptyset$, any other subset of the congruence generated by the identities of the variety, we still get the free algebra.
        
        At the opposite end, let $e_\omega$ be an index for $\omega$. Then $\V{n}[e_\omega] = T(X_n)/\approx_\omega$, where $\approx_\omega$ is the full congruence relation. That is, $\V{n}[e_\omega]$ is trivial, i.e.\ of size $1$.
    \end{example}
	
	When working with presentations, the following two lemmas turn out to be useful. They show that adding identities to a (not necessarily c.e.) presentation corresponds to quotienting and also, in a certain sense, vice versa. They formalize a fundamental intuition --- therefore, we will later often use them but without mentioning them explicitly.
	\begin{lemma}\label{adding_equations}
		Let $\V{}$ be a variety, $X$ a set, and $R, S$ sets of equations of terms over $X$, with $R \subseteq S$. Then $\langle X \mid  R\rangle_\V{} \twoheadrightarrow \langle X \mid  S\rangle_\V{}$.

	\end{lemma}
	\begin{proof}
		This is the third isomorphism theorem (Theorem \ref{third_isom_theorem}), stated for presentations. If $\approx_S, \approx_R$ are the congruence relations induced by $S, R$ respectively, then $\T(X)/\approx_S$ is isomorphic to the quotient of $\T(X_n)/\approx_R$ by $E:= \{([s]_{\approx_R},[t]_{\approx_R})\mid  s \approx_S t\}$. 
    \end{proof}

	\begin{lemma}\label{quotient_superset}
		\begin{enumerate}
			\item Let $\V{}$ be a variety and let $\mathcal{A},\mathcal{B} \in \V{}$, s.t.\ $\A \twoheadrightarrow \B$. Then for all presentations $\langle X \mid  R \rangle_\V{}$ of $\mathcal{A}$ there exists a $S \supseteq R$ s.t.\ $\langle X \mid  S \rangle_\V{}$ is a presentation of $\mathcal{B}$. 
			\item If, furthermore, the homomorphism witnessing $\A \twoheadrightarrow \B$ is non-injective (equivalently: $\B$ is isomorphic to a non-trivial quotient of $\A$), then an $S$ exists s.t.\ the inclusion is strict, i.e.\ $S \supsetneq R$.
		\end{enumerate}

	\end{lemma}
	\begin{proof}
		\begin{enumerate}
			\item Let $\langle X \mid  R \rangle$ be a presentation of $\mathcal{A}$ and let $\approx_R$ be the congruence relation generated by $R$ and the identities of the variety. Therefore $\mathcal{A} \cong \T(X)/\approx_R$. Since $\mathcal{B}$ is a homomorphic image of $\mathcal{A}$, there exists a congruence relation $\approx$, s.t. \[\mathcal{B} \cong (\T(X)/\approx_R)/\approx\]
			We define $S:= \{(s,t) \in \T(X) \times \T(X)\mid  [s]_{\approx_R} \approx [t]_{\approx_R}\}$. By reflexivity of $\approx$, $S \supseteq \approx_R$ and therefore also $S \supseteq R$. Furthermore, $S$ is a congruence relation on $\T(X)$: reflexivity, symmetry, and transitivity are easy to check and to show compatibility, observe that for any $n$ and any $n$-ary function symbol $f$,
			\begin{align*}
				&s_1~S~t_1, ...,s_n~S~t_n \implies [s_1]_{\approx_R} \approx [t_1]_{\approx_R}, ...,[s_n]_{\approx_R} \approx [t_n]_{\approx_R} \implies \\ &\implies f^{\T(X)/\approx_R}([s_1]_{\approx_R},..., [s_n]_{\approx_R}) \approx f^{\T(X)/\approx_R}([t_1]_{\approx_R},..., [t_n]_{\approx_R}) \implies \\
				&\implies [f(s_1,...,s_n)]_{\approx_R} \approx [f(t_1,...,t_n)]_{\approx_R} \implies f(s_1,...,s_n) ~S~ f(t_1,...,t_n).
			\end{align*}

			It remains to be shown that $\T(X)/S$ is isomorphic to  $(\T(X)/\approx_R)/\approx$ and therefore isomorphic to $\mathcal{B}$. 
			
			We show that $\varphi: \T(X)/S \to (\T(X)/\approx_R)/\approx,[t]_{S} \mapsto [[t]_{\approx_R}]_{\approx}$ is an isomorphism. 
			\begin{itemize}
				\item well-defined and injective: \begin{align*}
				    [s]_{S} = [t]_{S} \iff s S t \iff [s]_{\approx_R} \approx [t]_{\approx_R} \iff [[s]_{\approx_R}]_{\approx} = [[t]_{\approx_R}]_{\approx}.
                    \end{align*}
				\item surjective: follows immediately from the definition.
				\item homomorphism: 
                \begin{align*}
                                    &\varphi(f^{\T(X)/S}([s_1]_{S},...,[s_n]_{S})) = \varphi([f(s_1,...,s_n)]_{S}) = [[f(s_1,...,s_n)]_{\approx_R}]_{\approx} \\
                                    =& [f^{\T(X)/\approx_R}([s_1]_{\approx_R},..., [s_n]_{\approx_R})]_{\approx}= f^{(\T(X)/\approx_R)/\approx}([[s_1]_{\approx_R}]_{\approx},..., [[s_n]_{\approx_R}]_{\approx}) \\
                                    =& f^{(\T(X)/\approx_R)/\approx}(\varphi([s_1]_{S}),...,\varphi([s_n]_{S})).
                                                    \end{align*}

			\end{itemize}

			\item In the proof above, note that if $\approx$ is non-trivial, there exist $s,t \in \T(X)$ s.t. $[s]_{\approx_R} \neq [t]_{\approx_R}$ and $[s]_{\approx_R} \approx [t]_{\approx_R}$. Therefore, $(s,t) \in S$, but $(s,t) \notin \approx_R$ and thus also $(s,t) \notin R$.
		\end{enumerate}
	\end{proof}

    At one point in this paper (Proposition \ref{hom_open_implies_not_below_eqce}), we also need an ``effective" version of Lemma \ref{quotient_superset} for finitely generated presentations. This one is a bit less obvious. 
    \begin{lemma}\label{effective_quotient_superset}
        Let $\V{}$ be a variety and let $\mathcal{A},\mathcal{B} \in \V{\fg}$ be c.e.\ algebras s.t.\ $\A \twoheadrightarrow \B$. Then for all c.e.\ presentations $\langle X \mid W_e \rangle_\V{}$, with $X$ finite, of $\mathcal{A}$ there exists a c.e.\ set $W_i \supseteq W_e$ s.t.\ $\langle X \mid  W_i \rangle_\V{}$ is a presentation of $\mathcal{B}$. 
    \end{lemma}
    \begin{proof}
        Let $\langle X,W_j\rangle_\V{}$ be a presentation of $\B$, which we know exists because $\B$ is in $\V{\fg}$ and $\B$, as a homomorphic image of $\A$, cannot have more generators than $\A$. Let $\approx_{W_e}, \approx_{W_j}$ be the congruence relations generated by $W_e,W_j$, respectively. Observe that a surjective homomorphism $\varphi$ witnessing $\T(X)/\approx_{W_e} \twoheadrightarrow \T(X)/\approx_{W_j}$ is induced by a mapping of finitely many generators. The same mapping also induces a surjective homomorphism from $\T(X)$ to $\T(X)/\approx_{W_j}$, let us call it $\varphi^\prime$. For a term $t \in \T(X)$, finding a representative term of $\varphi^\prime(t)$ is computable, since it only requires the finite information of where generators are mapped. Therefore, $\ker(\varphi^\prime)$ is c.e.\ since we can, for each pair of terms, follow $\varphi^\prime$ and then solve the c.e.\ word problem of $\langle X, W_j\rangle_\V{}$. But now notice that $S:= \{(s,t) \in \T(X) \times \T(X)\mid [s]_{\approx_{W_e}} \approx_{\ker(\varphi)} [t]_{\approx_{W_e}}\}$ is equivalent to $\ker(\varphi^\prime)$. And we have shown in the proof of Lemma \ref{quotient_superset} that $S$ is a congruence relation and that $\T(X)/S$ is isomorphic to $\B$. Thus, for an index $i$ of the c.e.\ set $\ker(\varphi^\prime) = S \supseteq W_e$, $\langle X \mid  W_i \rangle_\V{}$ is a presentation of $\mathcal{B}$.
    \end{proof}

	\subsection{Isomorphism relations on c.e.\ presentations}
    In algebra, the isomorphism problem for a given variety $\V{}$, which we denote by $\IP(\V{})$, refers to the isomorphism problem for finite presentations in $\V{}$. 
    \begin{observation}[Folklore]\label{fp_sigma_1}
        For any variety $\V{}$, $\IP(\V{})$ is $\Sigma^0_1$.
    \end{observation}
    \begin{proof}
        Let $\langle X \mid R\rangle$ and $\langle Y \mid S\rangle$ be finite presentations of algebras in  $\V{}$. Let $\approx_R$ and $\approx_S$ be the congruence relations generated by $R$ and $S$, respectively (together with the equational laws of $\V{}$).
        We claim that $\T(X)/\approx_R \cong \T(Y)/\approx_S$ iff
        there exist mappings $\varphi: X \to \T(Y)$ and $\psi: Y \to \T(X)$, s.t.\ the following conditions hold:
        \begin{align}
            \bigwedge_{x \in X} \psi^\prime(\varphi(x)) \approx_R x,\\
            \bigwedge_{y \in Y} \varphi^\prime(\psi(y)) \approx_S y,\\
            \bigwedge_{(r_1,r_2) \in R} \varphi^\prime(r_1) \approx_S \varphi^\prime(r_2),\\
            \bigwedge_{(s_1,s_2) \in S} \psi^\prime(s_1) \approx_R \psi^\prime(s_2),
        \end{align}
    where $\varphi^\prime, \psi^\prime$ are the extensions of $\varphi,\psi$ to $\T(X), \T(Y)$, defined by (analogous for $\psi^\prime$)
    \begin{itemize}
        \item $\varphi^\prime(x) = \varphi(x)$ for $x \in X$,
        \item $\varphi^\prime(f(t_1,...,t_n)) = f(\varphi^\prime(t_1),...,\varphi^\prime(t_n))$, for terms $t_1,...,t_n \in \T(X)$ and function symbols $f$ of arity $n$ in the signature of $\V{}$.
    \end{itemize}
    
    ``$\implies$"-direction of the claim: Let $\phi$ be an isomorphism from $\T(X)/\approx_R$ to $\T(Y)/\approx_S$. It is immediate that $\phi\restriction X$ and $\phi^{-1}\restriction Y$ satisfy conditions (1)-(4).
    
    ``$\impliedby$"-direction of the claim: We show that $\varphi^\prime$ is a homomorphism from $\T(X)/\approx_R$ to $\T(Y)/\approx_S$, and $\psi^\prime$ is a homomorphism from $\T(Y)/\approx_S$ to  $\T(X)/\approx_R$, and they are each other's inverse functions. From that it follows immediately that $\varphi^\prime$ is also an isomorphism. In order to show that $\varphi^\prime$ (and analogously $\psi^\prime$) is a homomorphism, it is sufficient to check well-definedness on the equivalence classes of $\approx_R$. The homomorphism condition then follows from the definition of $\varphi^\prime$. So assume that for terms $t_1,t_2 \in T(X)$, $t_1 \approx_R t_2$. This equivalence can be derived from a set of equations $\hat{R} \subseteq R$ through symmetry, transitivity, and function application. By condition (3), for each $(r_1,r_2) \in \hat{R}$, it holds that $\varphi^\prime(r_1) \approx_S \varphi^\prime(r_2)$. Thus, we are also able to derive $\varphi^\prime(t_1) \approx_S \varphi^\prime(t_2)$ by replacing in the derivation of $t_1 \approx_R t_2$ every occuring term by its image under $\varphi^\prime$. This shows that $\varphi^\prime$ is well-defined. The fact that $\psi^\prime$ is the inverse of $\varphi^\prime$ (and analogously vice versa) can be shown by a simple induction on terms $t \in \T(X)$. The base case $t = x \in X$ follows from condition (1). For the step case $t = f(t_1,...,t_n)$, observe that $\psi^\prime(\varphi^\prime(f(t_1,...,t_n))) = f(\psi^\prime(\varphi^\prime(t_1)),...,\psi^\prime(\varphi^\prime(t_n)))$ which by induction hypothesis is $\approx_R$-equivalent to $f(t_1,...,t_n)$. This completes the proof of the claim.

    Finally, checking conditions (1)-(4) is $\Sigma^0_1$, because checking membership in $\approx_R, \approx_S$, i.e.\ the word problem, is $\Sigma^0_1$ (one can computably enumerate all equations that can be derived from the equations in $R$ or $S$). The mappings $\varphi$ and $\psi$ are finite objects, so asking whether they exist is just another existential quantifier. Therefore, the existence of an isomorphism is a $\Sigma^0_1$-question.
    \end{proof}

Moving to isomorphisms between c.e.\ presentations, we define isomorphism relations on the classes we introduced in Section \ref{sec:ce_pres}.
	\begin{definition}
		Let $\mathcal{V}$ be a variety and let $e, i \in \omega$. We define the corresponding \emph{isomorphism relations} on $\mathcal{V}$, $\mathcal{V}_n$, and $\mathcal{V}_{f.g.}$ as follows:
		\[
		e \iV{} i :\iff \mathcal{V}[e] \cong \mathcal{V}[i],
		\]
		\[
		e \iV{n} i :\iff \mathcal{V}_n[e] \cong \mathcal{V}_n[i], \text{ and}
		\]
		\[
		e \iV{\fg} i :\iff \mathcal{V}_{f.g.}[e] \cong \mathcal{V}_{f.g.}[i]
		\]

	\end{definition}

	Note that any such isomorphism relation is an equivalence relation on $\omega$. Also, each of them is a \emph{quotient} of $\eqce$, meaning that all equivalences from $\eqce$ hold (i.e.\ it is a superset of $\eqce$). 

    Analogous to Observation \ref{fp_sigma_1}, observe that for finitely generated algebras, the isomorphism relation can be at most $\Sigma^0_3$.
    \begin{observation}[Folklore]\label{fg_sigma_3}
        For any variety $\V{}$, $\iV{\fg}$ is $\Sigma^0_3$. 
    \end{observation}
    \begin{proof}
        Consider any $\langle k,e\rangle$ and $\langle j, i\rangle$, coding presentations $\langle X_k \mid W_e\rangle$ and $\langle X_j \mid W_i \rangle$. As in the proof of Observation \ref{fp_sigma_1}, we can ask for the existence of the mappings $\varphi: X_k \to \T(X_j)$ and $\psi: X_j \to \T(X_k)$. Since our generating sets are still finite, we can keep conditions (1) and (2) from that proof. Conditions (3) and (4) need to be changed to the following:
        \begin{align}
            (\forall r_1,r_2) ((r_1,r_2) \in W_e \implies \varphi^\prime(r_1) \approx_{W_i} \varphi^\prime(r_2)),\tag{3}\\
            (\forall s_1,s_2) ((s_1,s_2) \in W_i \implies \psi^\prime(s_1) \approx_{W_e} \psi^\prime(s_2)).\tag{4}
        \end{align}
The verification that this characterizes the existence of an isomorphism works analogously. Conditions (3) and (4) are now $\Pi^0_2$-statements, so the existence of $\varphi$ and $\psi$ that satisfy these conditions is now a $\Sigma^0_3$-statement.
    \end{proof}

\subsection{Computable reducibility and equivalence relations on c.e.\ sets}\label{sec: comp red}
Let $E$, $F$ be equivalence relations on $\nat$. Then $E$ is \emph{computably reducible} to $F$, denoted $E \leq_c F$, if there exists a computable function $h: \nat \to \nat$, s.t.~for all $x,y \in \nat$
\[
x ~E~ y \iff h(x) ~F~ h(y)
\]
If both $E \leq_c F$ and $F \leq_c E$, then $E$ and $F$ are \emph{computably bi-reducible} (or \emph{computably equivalent}), denoted $E \equiv_c F$.

Next, we mention some important equivalence relations \emph{on c.e.\ sets}. These are equivalence relations defined on indices of c.e.\ sets, namely, if two indices refer to the same set, they are in the same equivalence class. 

\begin{definition}[\cite{coskey2012hierarchy}] For $e,i \in \nat$ we define
	\begin{itemize}
		\item $e \eqce i \iff W_e = W_i$,
		\item $e ~\emin~ i \iff \min(W_e) = \min(W_i) \text{~or~} W_e = W_i = \emptyset$,
		\item $e ~\emax~ i \iff \max(W_e) = \max(W_i), W_e = W_i = \emptyset, \text{~or~} W_e, W_i \text{~are both infinite}$,
		\item $e ~\ezero~ i \iff \text{the symmetric difference } W_e \triangle W_i \text{ is finite}$.
	\end{itemize}

\end{definition}

\begin{theorem}[\cite{coskey2012hierarchy}]\ 
	\begin{itemize}
        \item $\emin$ and $\emax$ are incomparable under computable reducibility,
		\item $\emin<_c \eqce$, $\emax <_c \eqce$, and
		\item $\eqce <_c \ezero$.
	\end{itemize}
\end{theorem}

One important dividing line in this paper will be whether or not a given isomorphism relation is below $\eqce$. Note that even though $\eqce$ is $\Pi^0_2$-complete under $m$-reducibility, it is not complete under computable reducibility. In fact, such an equivalence relation cannot exist.

\begin{theorem}[\cite{ianovski2014complexity}]\label{no_universal_pi2} There is no $\Pi^0_2$-complete equivalence relation (under $\leq_c$). In fact, for each equivalence relation $R \in \Pi^0_2$, there is an equivalence relation $S \in \Delta^0_2$, s.t.\ $S \not \leq_c R$.
\end{theorem}

For us, this means that showing that an isomorphism relation is $\Pi^0_2$ and even $\Delta^0_2$ is not sufficient to determine reducibility to $\eqce$.

Next, we introduce the monotonicity lemma \cite{coskey2012hierarchy}, which will allow us to more easily show certain non-reducibility results by diagonalization. Hidden in the proof of the lemma is an application of the recursion theorem, which we otherwise would have to explicitly use for each non-reducibility proof.

\begin{lemma}[Monotonicity Lemma, \cite{coskey2012hierarchy}]\label{monotonicity_lemma}
	Suppose $f: \nat \to \nat$ is computable and \emph{well-defined on c.e.\ sets}, i.e.\ for every $e,i$, $W_e = W_i \implies W_{f(e)} = W_{f(i)}$. Then for every $e,i\in\nat$, $W_e \subseteq W_i \implies W_{f(e)} \subseteq W_{f(i)}$.
\end{lemma}
Note that the monotonicity lemma applies to any computable $f$ that reduces some quotient of $\eqce$ to $\eqce$.

\subsection{\texorpdfstring{$\emin$}{Emin} on well-orderings}

Coskey, Hamkins, and Miller \cite{coskey2012hierarchy} generalized $\emin$ and $\emax$ to any computable linear ordering of the natural numbers, using Dedekin cuts. For us, it will be sufficient to consider $\emin$ on computable well-orderings. This allows us to use a simpler definition that is equivalent to theirs for the cases that we consider. 

\begin{definition}
    Let $\A$ be a computable well-ordering of $\nat$. We define $\emin(\A)$ by $e ~\emin(\A)~ i$ if and only if $W_e$ and $W_i$ both have the same $<_\A$-least element or are both empty. 
\end{definition}

These $\emin$-equivalence relations have the following properties: 
\begin{enumerate}
    \item for any computable well-ordering $\A$, $\emin(\A) \le_c \eqce$, 
    \item there are two different computable well-orderings $\A, \B$ of the same order type (ordinal), where $\emin(\A)$ and $\emin(\B)$ are not computably bi-reducible, and 
    \item for any two computable well-orderings $\A, \B$ of different order type, $\emin(\A)$ and $\emin(\B)$ are not computably bi-reducible.
\end{enumerate}

(1) and (3) directly follow from analogous results that \cite{coskey2012hierarchy} have shown in the context of their definition. With Proposition \ref{c_red_implies_zero_jump_emb} and Proposition \ref{emin_different_ordinals}, we prove (2) and re-prove (3), respectively, in our context and notation. For that, we use the following two simple lemmas about properties preserved by a computable reduction that is well-defined on c.e.\ sets.

\begin{lemma}\label{f_preserves_order}
    Let $\A,\B$ be computable well-orderings and let $f$ be a computable reduction from $\emin(\A)$ to $\emin(\B)$ that is well-defined on c.e.\ sets. Then $\min_\A(W_e) \leq_\A \min_\A(W_i)$ implies $\min_\B(W_{f(e)}) \leq_\B \min_\B(W_{f(i)})$.
\end{lemma}
\begin{proof}
    Let $a_e := \min_\A(W_e)$ and $a_i := \min_\A(W_i)$.
    The reduction $f$ maps $W_i$ and $\{a_i\}$ into the same $\emin(\B)$-class. Furthermore, $f$ maps $W_e$, $\{a_e\}$, and $\{a_e, a_i\}$ into the same $\emin(\B)$-class.
    By the monotonicity lemma (Lemma \ref{monotonicity_lemma}), $f(\{a_i\}) \subseteq f(\{a_e, a_i\})$, thus $\min_\B(f(a_e, a_i\})) \leq_\B \min_\B(f(\{a_i\}))$ and therefore also $\min_\B(W_{f(e)}) \leq_\B \min_\B(W_{f(i)})$.
\end{proof}

Note that, since reductions preserve equivalence classes, we also have that if $\min_\A(W_e) <_\A \min_\A(W_i)$ then $\min_\B(W_{f(e)}) <_\B \min_\B(W_{f(i)})$.

For a well-ordering $\A$, let the $\A$\emph{-rank} (or just \emph{rank}, if $\A$ is clear from the context) of an element $a$ in $\A$ be the order type of the initial segment $\{b \in \A\mid b <_\A a\}$.
\begin{lemma}\label{f_preserves_rank}
    Let $\A,\B$ be computable well-orderings and let $f$ be a computable reduction from $\emin(\A)$ to $\emin(\B)$ that is well-defined on c.e.\ sets.
    Let $a_\gamma$ be the $\A$-least element of $W_e$, and $\gamma$ the $\A$-rank of $a_\gamma$. Then the $\B$-least element of $W_{f(e)}$ exists and has $\B$-rank at least $\gamma$. 
\end{lemma}
\begin{proof}
    We show this by transfinite induction on $\gamma$.
    \begin{itemize}
        \item Base case: Assume $\gamma = 0$. So $W_e$ contains an element $a_0$ of $\A$-rank $0$, thus is not empty. Let $i$ be an index of the empty set. By the monotonicity lemma (Lemma \ref{monotonicity_lemma}), $W_i \subseteq W_e$ implies $W_{f(i)} \subseteq W_{f(e)}$. Furthermore, since $W_i$ and $W_e$ are in different $\emin(\A)$-classes, we have $W_{f(i)} \subsetneq W_{f(e)}$. Thus, $W_{f(e)}$ cannot be empty and therefore has a least element of $\B$-rank at least $0$.
        \item Successor case: Assume $\gamma$ is a successor of $\delta$ and that the statement holds for $a_\delta$ of $\A$-rank $\delta$. Observe that $\min_\A(W_e \cup a_\delta) <_\A \min_\A(W_e)$. We can apply the induction hypothesis to $W_e \cup a_\delta$ and get that its image under $f$ has as its $\B$-minimum an element of $\B$-rank at least $\delta$. By Lemma \ref{f_preserves_order}, $\min_\B(W_{f(e)})$ must then have rank strictly greater than $\delta$, thus at least $\gamma$.

        \item Limit case: Assume $\gamma$ is a limit ordinal and assume that the statement holds for all $a_\delta$ of $\A$-rank $\delta < \gamma$. We can apply the induction hypothesis to any $W_e \cup a_\delta$, for $\delta < \gamma$ to get that the $\B$-least element of $f(W_e \cup a_\delta)$ has $\B$-rank at least $\delta$. Analogously to the successor case we then conclude by Lemma \ref{f_preserves_order} that $\min_\B(f(W_e))$ must have $\B$-rank strictly above any $\min_\B(f(W_e \cup a_\delta))$, $\delta < \gamma$, therefore at least the limit of their $\B$-ranks, which is $\gamma$.
    \end{itemize}
\end{proof}

\begin{observation}\label{emin_red_well_def}
    Let $\A$ be a computable well-ordering and let $E$ be an equivalence relation on c.e.\ sets. If $E$ is computably reducible to $\emin(\A)$, then there is a computable reduction from $E$ to $\emin(\A)$ that is well-defined on c.e.\ sets.
\end{observation}
\begin{proof}
    Let $f$ be a computable reduction from $E$ to $\emin(\A)$. We define the computable function $f^\prime$ on indices of c.e.\ sets as follows: for any c.e.\ set $W_e$, enumerate into $W_{f^\prime(e)}$
    \begin{itemize}
        \item any element that is enumerated into $W_{f(e)}$ and
        \item for any element $m$ that is enumerated, we also enumerate any $n >_\A m$ (which we can do because $\A$ is computable).
    \end{itemize} 
    For any $e$, $W_{f(e)}$ and $W_{f^\prime(e)}$ are $\emin(\A)$-equivalent. Thus, $f^\prime$ is also a reduction. Since $E$ is an equivalence relation on c.e.\ sets, $W_e = W_i$ implies $eEi$, which implies $f(e)\emin(\A)f(i)$. Thus, $W_{f(e)}$ and $W_{f(i)}$ have the same $\A$-least element. By our construction, that implies $W_{f^\prime(e)} = W_{f^\prime(i)}$. So $f^\prime$ is well-defined on c.e.\ sets.
\end{proof}

\begin{proposition}\label{c_red_implies_zero_jump_emb}
    Let $\A,\B$ be computable well-orderings and let $\emin(\A) \leq_c \emin(\B)$. Then there is an embedding of $\A$ into $\B$ that is computable from $\mathbf{0^\prime}$.
\end{proposition}
\begin{proof}
    By Observation \ref{emin_red_well_def}, let $f$ be a computable reduction that is well-defined on c.e.\ sets. Define an embedding $g$ of $\A$ into $\B$ by
    \[
        g(a) := \min_\B(f(\{a\}).
    \]
    To verify that $g$ is an embedding, observe that, by Lemma \ref{f_preserves_order}, $a_1 \leq_\A a_2$ implies $\min_\B(f(\{a_1\}) \leq_\B \min_\B(f(\{a_2\})$.
    Furthermore, it is easy to see that one Turing jump can compute $g$.
\end{proof}

Bazhenov, Rossegger, and Zubkov \cite{bazhenov2022bi} have shown in particular that there are computable well-orderings $\A$ and $\B$ of order type $\omega^2$ s.t.\ there is no embedding of $\A$ into $\B$ that is computable from $\mathbf{0}^{\prime\prime}$. Therefore, Proposition \ref{c_red_implies_zero_jump_emb} shows that there are computable well-orderings of the same order type, whose corresponding $\emin$-equivalence relations are not computably bi-reducible.

\begin{proposition}\label{emin_different_ordinals}
    Let $\beta < \alpha$ be two computable ordinals and let $\B$ be a computable well-ordering of type $\beta$ and $\A$ a computable well-ordering of type $\alpha$. Then $\emin(\A) \not\leq_c \emin(\B)$.
\end{proposition}
\begin{proof}
    By Observation \ref{emin_red_well_def}, let $f$ be a computable reduction that is well-defined on c.e.\ sets. Let $a_\gamma \in \A$ be an element of $\A$-rank $\gamma > \beta$. But we know by Lemma \ref{f_preserves_rank} that $\min_\B(W_{f\{a_\gamma\}})$ has $\B$-rank at least $\gamma$, a contradiction.
\end{proof}

In \cite{coskey2012hierarchy}, there is no distinction between the ordinal and the actual well-ordering, even though different well-orderings of the same ordinal can be in different $c$-degrees. This does not make a difference for their results, but we need to be more precise for our purposes.
We are particularly interested in well-orderings that are ``computable" in an even stronger sense, namely for any element we can computably know its rank, i.e.\ we know exactly where in the order it is. Therefore, any two such orderings of the same ordinal are then also computably isomorphic, and the associated $\emin$-relations are computably bi-reducible. Thus, we can define $\emin(\alpha)$ using these ``strongly computable" well-orderings, and its $c$-degree will be well-defined. Also, by Proposition \ref{emin_different_ordinals}, we know that the degree of some $\emin(\alpha)$ cannot contain an $\emin(\B)$ for a well-ordering $\B$ of an order type different from $\alpha$. Later, we will see natural examples of isomorphism relations that fall in the degrees of some $\emin(\alpha)$, demonstrating the utility of this definition.

We give a canonical definition of an encoding up to $\omega^\omega$, since that is the highest we need to go for our work. Note that by using Cantor normal form, one could also define an encoding up to $\epsilon_0$ (the supremum of $\{1,\omega, \omega^\omega, \omega^{\omega^\omega},...\})$.    
\begin{definition}
    Let $\alpha$ be a computable ordinal up to (including) $\omega^\omega$. That is, each element in it has normal form $\alpha = \omega^{e_1}c_1 + ..., \omega^{e_k}c_k$ with $e_1,...,e_k,c_1,...,c_k \in \omega$ and $e_1 > ... > e_k$ (this is essentially Cantor normal form below $\omega^\omega$). Let $h$ be a bijective and computable encoding of all ordinals (represented as finite tuples $\langle e_1,...,e_k,c_1,...,c_k \rangle$) below $\alpha$ into $\omega$ and let $\A_\alpha$ be the ordering defined by $n \leq_{\A_\alpha} m \iff h^{-1}(n) < h^{-1}(m)$. Then we define $\emin(\alpha) := \emin(\A_\alpha)$.
\end{definition}

\begin{proposition}\label{E_min_strictly_ascending}
    If $\alpha < \beta$ then $\emin(\alpha) <_c \emin(\beta)$.
\end{proposition}
\begin{proof}
    ``$\emin(\alpha) \leq_c \emin(\beta)$": We define a computable reduction $f$ s.t.\ \[e ~\emin(\alpha)~ i \iff f(e) ~\emin(\beta)~ f(i).\] Given an index $e$, we provide an effective procedure to enumerate $W_{f(e)}$: for any element $n \in \omega$ entering $W_e$ we calculate its $\A_\alpha$-rank $\gamma$ and enumerate an $m$ of $\A_\beta$-rank $\gamma$ into $W_{f(e)}$. Then any two $W_e, W_i$ have the same $\A_\alpha$-least element iff $W_{f(e)}, W_{f(i)}$ have the same $\A_\beta$-least element.

    ``$\emin(\beta) \not\leq_c \emin(\alpha)$": Follows from Proposition \ref{emin_different_ordinals}.
\end{proof}

\subsection{Classes of algebras}
In what follows, we will investigate $n$-generated and finitely generated classes in the following varieties:
\begin{itemize}
    \item $\AG{}$: the abelian (commutative) groups,
    \item $\UF{m}{}$: all algebras with exactly $m$ function symbols and all of them unary (of arity 1),
    \item $\CS{}$: the commutative semigroups, and
    \item $\CM{}$: the commutative monoids.
\end{itemize}

At this point, we also mention an example of two very simple commutative semigroup presentations that are isomorphic but not via a permutation of generators. This is to illustrate the difficulty of recognizing isomorphic presentations.

\begin{example}\label{ex:non_trivial_isomorphism}
    Consider the presentations $\langle a,b \mid a = a^2b^2\rangle$ and $\langle c,d \mid c = c^2d^3\rangle$. Clearly, neither $c\mapsto a, d\mapsto b$ nor $c\mapsto b, d\mapsto a$ can induce an isomorphism. However, these two presentations are still isomorphic: take the mappings $\varphi: a \mapsto cd, b \mapsto d$ and $\psi: c \mapsto a^2b, d \mapsto b$ and check the conditions given in Observation \ref{fp_sigma_1}.
\end{example}

Still, it is known that for finite presentations, this is a decidable problem.
\begin{theorem}[\cite{taiclin1974isomorphism}, \cite{grunewald1980some}]\label{isomorphism_problem_decidable}
	The isomorphism problem is decidable for finite presentations of commutative monoids. 
\end{theorem}

\section{Isomorphism relations in the \texorpdfstring{$\emin$}{Emin}-degrees}\label{sec: isorelations and emin}

In this section, we show that $\iAG{1}$, $\iUF{1}{1}$, $\iCS{1}$, and $\iCM{1}$ are all computably equivalent to $\emin$, and that $\iAG{n}$ and $\iUF{1}{n}$ are computably equivalent to $\emin(\omega \cdot n)$.

	\subsection{Monogenic semigroups and monoids}
	With a monogenic semigroup $\A$ (generated by $X_1 = \{x\}$) one usually associates an \emph{index} and a \emph{period}, where the index is the least $n$, s.t. $x^n = x^\ell$ for some $\ell \neq n$, and the period is $k-n$ for the least $k\neq n$ with $x^n = x^k$. It is a well-known result that index and period together characterize the isomorphism types.
	
	\begin{proposition}[folklore]
		Two monogenic semigroups $\A,\B$ are isomorphic iff they have the same index and period.
	\end{proposition}

	The following observations will also be useful. They state that (1) neither index nor period increases with additional identities and (2) for a given finite set of identities, one can compute the index and period.

	\begin{observation} \label{not_increase}
		In monogenic semigroups, neither index nor period increases when taking a quotient.
	\end{observation}
\begin{proof}
    Clearly, adding further identities can never increase the index. So it remains to be shown that the period does not increase. Let $\A \twoheadrightarrow \B$, let $i_A,p_A$ be the index and period of $\A$ and let $i_B,p_B$ be the index and period of $\B$. We will show that $x^{i_B} = x^{i_B+p_A}$ in $\B$, which implies that the period of $\B$ is at most $p_A$.
	
	 From $x^{i_B} = x^{i_B + p_B}$, due to compatibility and transitivity of congruence relations, it follows that
\begin{align}
	 	&x^{i_B} = x^{i_B+p_B} = x^{i_B+2p_B} = ...\tag{1}\\ 
	 	&x^{i_B+p_A} = x^{i_B+p_A+p_B} = x^{i_B+p_A+2p_B} = ... \tag{2}
\end{align}
  We now pick a $k$ s.t. $i_B+kp_B \geq i_A$. Then, from $x^{i_A} = x^{i_A+p_A}$ (which holds because $x^{i_A} = x^{i_A+p_A}$ holds in $\A$ and $\B$ has all identities of $\A$), it follows by compatibility that $x^{i_B+kp_B} = x^{i_B+kp_B+p_A}$. Together with (1) and (2), we get, by transitivity, $x^{i_B} = x^{i_B+p_A}$.
	\end{proof}

\begin{observation}
	For any $e,s \in \nat$, index and period of $A = \langle x\mid  W_{e,s}\rangle_{\CS{1}}$ are computable. In particular, the index is the smallest $i$ so that $x^i = x^j$ or $x^j = x^i$ is a non-trivial relation, and the period is the greatest common divisor of all the $|i-j|$ where $x^i = x^j$ is a non-trivial relation. 
\end{observation}

With these observations, we are able to show that the isomorphism relation on monogenic semigroups is computably equivalent to $\emin$.

	\begin{theorem}\label{1iso_red_Emin}
	$\iCS{1} \equiv_c \emin$.
	\end{theorem}
\begin{proof}
	
	``$\leq_c$": We construct a computable $f:\nat \to \nat$, that maps $\iCS{1}$-instances (indices of c.e.\ sets, with the sets encoding presentations of monogenic semigroups) to $\emin{}$-instances (indices of c.e. sets), s.t.\
	\[
	e \iCS{1} i \iff f(e) \emin f(i) \iff \min(W_{f(e)}) = \min(W_{f(i)}) \text{~or~} W_{f(e)} = W_{f(i)} = \emptyset.
	\]
	Let $\CS{1}[e]$ be any monogenic semigroup. We build $W_{f(e)}$ in stages as follows:
	\begin{itemize}
		\item As long as $W_e$ stays empty, leave $W_{f(e)}$ empty as well.
		\item If an equation enters $W_e$, we calculate the index and period of the encoded monogenic semigroup, and put them, encoded as a pair, in $W_{f(x)}$. Index and period are computable because at each stage there are only finitely many elements already in $W_e$.
		\item Since, by Observation \ref{not_increase}, new equations cannot make the index and period increase, we will eventually stabilize to the true values for them.
	\end{itemize}
	Now any $W_e$, $W_i$ have the same index and period if and only if $W_{f(e)}$ and $W_{f(i)}$ have the same least element, encoding index and period (to be precise, we have to choose a paring function that descreases whenever one of the values decrease).

	``$\geq_c$": This time we construct a computable $f$ that maps $\emin{}$-instances to $\iCS{1}$-instances, s.t.\
    \[
	e \emin i \iff f(e) \iCS{1} f(i).
	\]
	
	From a c.e.\ set $W_e$, we construct $W_{f(e)}$ as follows:
	\begin{itemize}
		\item As long as $W_e$ stays empty, leave $W_f(e)$ empty as well (i.e.\ $\CS{1}[f(e)]$ is the free monogenic semigroup).
		\item  For some $n$ entering $W_e$ we put the code of $x^n = x^{n+1}$ in $W_{f(e)}$.
		\item Since we will eventually hit the least element $n$ of $W_e$, the index and period of $\CS{1}[f(e)]$ stabilizes to $\langle n, 1 \rangle$.
	\end{itemize}
  Now any two c.e.\ sets have the same least element $n$ if and only if the resulting monogenic semigroups are isomorphic.
\end{proof}

\begin{theorem}
    $\iCM{1} \equiv_c \emin$. 
\end{theorem}
The proof of this is identical to the proof for $\iCS{1}$, except now we have the identity behaving as the ``0-th'' power of the generator. We could also, in a very similar way, show that $\iAG{1}$ and $\iUF{1}{1}$ are in this degree too (it is not hard to show that $\iUF{1}{1}$ also admits a characterization by index and period). However, that follows from the more general case of $n$ generators, which we investigate in the remaining part of this section.

\subsection{\texorpdfstring{$\UF{1}{n}$}{UF1n}}\label{uf_section}

Next, we analyze the complexity of the classes $\UF{1}{n}$ of algebras with a single unary function symbol and show that they coincide with $\emin(\omega \cdot n)$. In preparation for this, we need some observations on the structure of algebras in $\UF1n$. We view these algebras through directed graphs where the edge relation is given by the unary function symbol $f$. Formally, given $\A = (A,f) \in \UF{1}{n}$, we define the \emph{graph associated with} $\A$ as $G(\A) := (A, E_\A)$, where $E_\A := \{(u,v) \in A \times A \mid  f(u) = v\}$.
\begin{observation}
    For a given $n$, let $\A, \B \in \UF{1}{n}$. Then, $\A \cong B$ iff $G(\A) \cong G(\B)$.
\end{observation}
\begin{proof}
It is sufficient to observe that the homomorphism condition 
\[
\varphi(f^\A(u)) = f^B(\varphi(u))
\]
for algebra isomorphisms $\varphi: \A \to \B$ is equivalent to the condition
\[
(u,v) \in E_{\A} \iff (\psi(u),\psi(v)) \in E_{\B}
\]
for isomorphisms $\psi: G(\A) \to G(\B)$ of directed graphs.
\end{proof}
For a given algebra $\A$, we note that two generators $a,b$ are in the same connected component of $G(\A)$ if and only if there are $i,j$ s.t.\ $f^i(a) = f^j(b)$ in $\A$.
In the following, we call an infinite sequence $v_0,v_1,...$ of distinct vertices, with an edge from $v_i$ to $v_{i+1}$ for all $i \geq 0$, a \emph{ray}. Furthermore, for a directed graph $G = (V, E)$ and a vertex $v \in V$, we define $G \restriction v$ to be the subgraph induced by $\{w \in V \mid v \text{~is reachable from~} w\}$. The following basic observations about the components will help us characterize them. 
\begin{observation}\label{uf_observations}
Let $\A \in \UF{1}{n}$. 
\begin{enumerate}
    \item Any vertex of $G(\A)$ has outdegree (edges leaving the vertex) of exactly $1$. 
    \item Infinite components of $G(\A)$ cannot contain a cycle.
    \item Let $G^\prime$ be an infinite component of $G(\A)$. Then, there exists a vertex $v\in G'$ s.t.\ $v$ is reachable from any generator and is \emph{the least such} in the sense that, for any vertex $u \neq v$ that is reachable from any generator, $v$ is not reachable from $u$. Also, the isomorphism type of $G^\prime$ is characterized by the finite graph $G^\prime \restriction v$.
    
\end{enumerate}
    
\end{observation}
\begin{proof}
    \begin{enumerate}
        \item For $f$ to be well-defined, any vertex must have outdegree $1$.
        \item Assume that a connected component has a cycle. Then there is a generator $a\in G'$ s.t.\ $f^i(a) = f^j(a)$ for some $i < j$. Any other generator in the component must have a path into the cycle. Therefore, the component is finite.
        \item For any two generators $a,b\in G'$, there exist $i,j$ s.t.\ $f^i(a) = f^j(b)$. So there exists a vertex reachable from any generator. $u$ can be chosen to be the least: if it is reachable from some $u$ that is reachable from any generator, we may replace $v$ by $u$ --- and this process must stop, as any $f$-predecessor chain must be finite. 
        Now, observe that $G^\prime \restriction v$ must be finite. Furthermore, the edges and vertices from $G^\prime$ that are not in $G^\prime \restriction v$ form exactly one single ray, originating at $v$. So assume we have two such graphs $G^\prime, F^\prime$ associated to two infinite components, and let $v$ and $w$ be the respective least vertices reachable from any generator (in the component). Then an isomorphism from $G^\prime \restriction v$ to $F^\prime \restriction w$ must send $v$ to $w$ (they are the only vertices without outgoing edges) and can therefore be extended to an isomorphism from $G^\prime$ to $F^\prime$. Thus, the isomorphism type of infinite components is characterized by these finite subgraphs.
    \end{enumerate}
\end{proof}

As a corollary of (1), note that for any vertex of $G(\A)$, there is a unique ray starting from that vertex. Figure \ref{fig:uf_graph_example} shows an example of a finite and an infinite component and how the finite graph characterizing an infinite component looks like.

\begin{figure}
\begin{center}
    \begin{tikzcd} 
          & f(f(x_2)) \arrow[bend right=20,ld]& & & &{\color{red}\vdots} &  \\
          f(f(f(x_2)))) \arrow[bend right=20,rd]&  & f(x_2) \arrow[bend right=20,lu]& & &{\color{red}f(f(x_6))} \arrow[u, color=red]& & & \\
          & f(f(x_1)) \arrow[bend right=20,ru]& x_2 \arrow[u]& & &f(x_6) \arrow[u,color=red]& & &\\
          & f(x_1) \arrow[u]& & &f(x_4) \arrow[ru]&f(f(x_5)) \arrow[u] & x_6 \arrow[lu]\\
          f(x_0) \arrow[ru]& x_1 \arrow[u]& &f(x_3) \arrow[ru] & x_4 \arrow[u] &f(x_5) \arrow[u]\\
          x_0 \arrow[u]& & &x_3 \arrow[u] & &x_5 \arrow[u] & x_7 \arrow[lu]\\
    
\end{tikzcd}
\end{center}
\caption{An example of an algebra in $\UF{1}{\fg}$. There are two components; the one on the left is finite and the one on the right is infinite. In the infinite component, the part that is not colored red (i.e.\ everything below and including $f(x_6)$) is the finite graph that characterizes the component.}
\label{fig:uf_graph_example}
\end{figure}
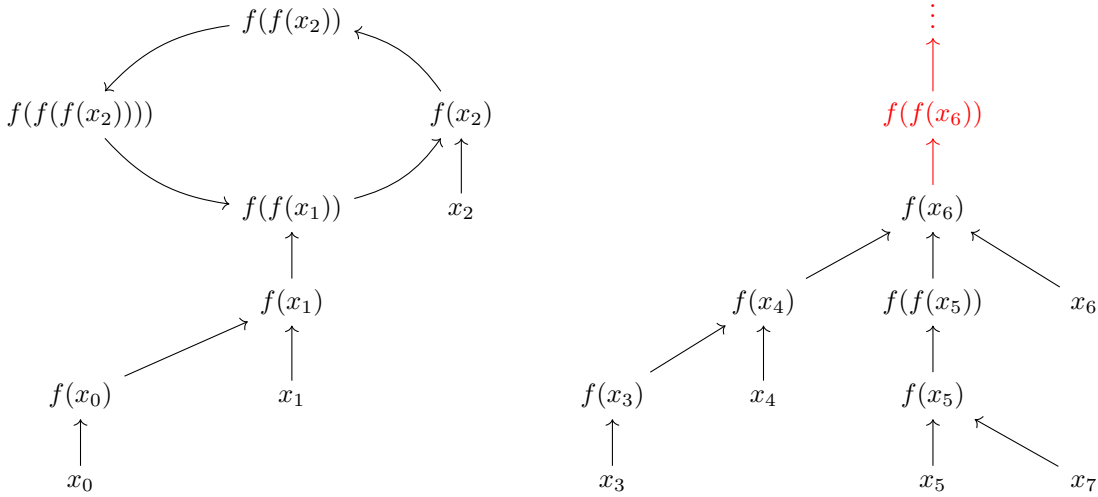

Based on these observations, we now give an invariant for algebras in $\UF{1}{n}$. First, we need a way of encoding finite directed graphs into natural numbers, s.t.\ if the number of vertices decreases, the code also decreases. For this, take an enumeration of all (isomorphism types of) finite directed graphs, starting with graphs of size $1$, then size $2$, and so on. 
Then, let $m(\A)$ be the number of infinite components in $G(\A)$ and let $I(\A)$ be a code for the finite graph that is $G(\A)$ but with each infinite component replaced by its characterizing finite subgraph, according to Observation \ref{uf_observations} (3). Notice that these finite graphs characterizing infinite components still cannot contain a cycle and thus will be different from the finite components in $G(\A)$. So, clearly $I(\A)$ would already characterize the isomorphism type, but we will take as invariant the pair of $m(\A)$ and $I(\A)$ because, as we will see, it interacts nicely with $\emin(\omega^2)$.

\begin{proposition}\label{uf_graph_computable}
Let $\A \in \UF{1}{n}$. Then, $m(\A)$ and $I(\A)$ can be computed from a finite presentation of $\A$. 
\end{proposition}
\begin{proof}
    To compute $m(\A)$, just keep track of which generators are in the same connected component. This can be done computably.
    
    For $I(\A)$, we show computability inductively on the finite set of equations in the presentation. If there are no equations, each of the $n$ components in $G(\A)$ is a ray. Thus, $I(\A)$ is $n$ copies of the graph that has a single vertex and no edges. For the induction step, assume that we have an invariant consisting of the finite connected components $I_0, ..., I_k$, each of them corresponding to a (not necessarily finite) connected component $G_0,..., G_k$ of $G(\A)$. Assume a (non-trivial) equation is added to the presentation, collapsing a vertex $u \in G_i$ with a vertex $v \in G_j$. We distinguish between the following cases: 
    \begin{enumerate}  
        \item $i \neq j$, both $G_i$ and $G_j$ are infinite,
        \item $i = j$, $G_i=G_j$ is infinite,
        \item $i \neq j$, both $G_i$ and $G_j$ are finite,
        \item $i = j$, $G_i=G_j$ is finite,
        \item $i \neq j$, $G_i$ is infinite, $G_j$ is finite. 
        
    \end{enumerate}  
    Observe that in each of these cases, after the collapse, the two collapsed vertices will be in the same component $G^\prime$. In $G(\A)$, $G^\prime$ replaces $G_i$ and $G_j$, while all other components stay the same. So it is sufficient to show that we can compute the invariant $I^\prime$ of $G^\prime$ from $I_i$ and $I_j$.
    
    If we know that $G^\prime$ is finite, we can compute it (and thus $I^\prime$) as follows: First, we find a cycle, by enumerating equalities between vertices until we see a cycle, i.e.\ we see that $f^k(x) = f^l(x)$ for some generator $x$ and $k,l \geq 0, k\neq j$. So we can look at the finite set of vertices consisting of the vertices generated by $x$ up to the cycle and the vertices generated by a different generator in the component up to where it connects to something $x$-generated. Then, it is easy to see that we can just exhaustively compute all collapses (all equations in the presentation and their consequences) between this set of vertices. So, this takes care of cases 3-5.

    For case 1, it is easy to see that $G^\prime$ is obtained from $G_i$ and $G_j$ by merging the two (unique) rays $u = u_0,u_1,...$ and $v = v_0,v_1,...$ starting from $u$ and $v$. So $I^\prime$ is obtained from $I_i$ and $I_j$ as follows.
    \begin{itemize}
        \item If $u$ and $v$ are already in $I_i$ and $I_j$ respectively, we just merge the parts of the rays that are in $I_i$ and $I_j$.
        \item Otherwise, if $u \notin I_i$, we first add the initial segment of the ray starting at $I_i$ up until (including) $u$. We similarly modify $I_j$ if $v \notin I_j$. Then, it is sufficient to just merge $u$ and $v$ (since at least one of the two does not have a successor).
    \end{itemize} 
    
    For case 2, two things can happen. If $u$ and $v$ are the same distance from the vertex that joins the two (unique) rays starting from $u$ and $v$, we can again just merge them up to that vertex. If not, we know that there will be a cycle, and we can proceed as in cases 3-5. Thus, we can for all cases compute the invariant that results from adding the equation.
\end{proof}

In the following, $<_\mathrm{lex}$ denotes the \emph{lexicographic order on pairs}, defined as usual by $(x_1,y_1) <_\mathrm{lex} (x_2,y_2) \iff x_1 < x_2 \lor (x_1 = x_2 \land y_1 < y_2)$.

\begin{observation}\label{obs:uf_graph_decreasing}
    Let $\A, \B \in \UF{1}{n}$, with $\A$ finitely presented by $\langle X_n \mid R \rangle$ and $\B$ finitely presented by $\langle X_n \mid S \rangle$ and with $\approx_R \subsetneq \approx_S$, where $\approx_R,\approx_S$ are the congruence relations generated by $R,S$, respectively. Then, $(m(\B),I(\B)) <_{\mathrm{lex}} (m(\A),I(\A))$.
\end{observation}
\begin{proof}
    W.l.o.g.\ we assume that $S$ contains exactly one additional non-trivial (i.e.\ not already implied by $R$) equation. From that, the observation follows by induction. 
    
    Assume that $I(\A)$ is a code for the graph consisting of the finite connected components $I_0, ..., I_k$, each of them corresponding to a (not necessarily finite) connected component $G_0,..., G_k$ of $G(\A)$. Assume the additional equation in the presentation of $\B$ collapses a vertex from $G_i$ with a vertex from $G_j$. 
    Again, we distinguish between the following cases: 
    \begin{enumerate}  
        \item $i \neq j$, both $G_i$ and $G_j$ are infinite,
        \item $i = j$, $G_i=G_j$ is infinite,
        \item $i \neq j$, both $G_i$ and $G_j$ are finite,
        \item $i = j$, $G_i=G_j$ is finite,
        \item $i \neq j$, $G_i$ is infinite, $G_j$ is finite. 
        
    \end{enumerate}  
    Let $G^\prime$ be the component that replaces $G_i$ and $G_j$ in $G(\B)$. So in $I(\A)$ the invariant $I^\prime$ of $G^\prime$ replaces $I_i$ and $I_j$.
    In cases 1 and 5, $m(\B)<m(\A)$. Cases 3 and 4 only involve finite components, so since $G^\prime$ will have fewer vertices than $G_i, G_j$, also $I^\prime$ must have fewer vertices than $I_i, I_j$. Therefore, $m(\B) = m(\A)$ and $I(\B) < I(\A)$. For case 2, notice that (as argued in the proof of Proposition \ref{uf_graph_computable}) either $I^\prime$ is obtained by merging (at least) two vertices, or there will be a cycle. If the former is true, $m(\B) = m(\A)$ and $I(\B) < I(\A)$. If the latter is true, $m(\B) < m(\A)$.
\end{proof}

\begin{proposition}\label{uf_emin}
$\iUF{1}{n} \equiv_c \emin(\omega \cdot n)$.
\end{proposition}
\begin{proof}
    ``$\geq_c$": To define a reduction $f$ from $\emin(\omega \cdot n)$ to $\iUF{1}{n}$, at each stage we compute the rank $\omega \cdot k + j$, with $k < n$, of the current least element of the $\emin(\omega \cdot n)$-instance and and put the equations
    \begin{itemize}
        \item $f^{j+1}(x_k) = f^{j}(x_k)$, and
        \item $f(x_i) = x_i$ for $i > k$
    \end{itemize}
    into the $\iUF{1}{n}$-instance.
    First, observe that if a new least element appears, the corresponding equations always already imply the equations that coded the previous least element. So the isomorphism type of the $\iUF{1}{n}$-instance is, at any stage, determined by the equations corresponding to the current least element in the $\emin(\omega \cdot n)$-instance.
    
    Second, observe that each component of the associated graph will always contain exactly one generator, so there will always be $n$ components. For a given $k$ and $j$, the components corresponding to the generators $x_0,...,x_{k-1}$ will be a ray. The component corresponding to $x_k$ consists of a path of length $j$ and a cycle of size $1$ at the end. The other components consist of just a cycle of size $1$. 
    
    Any two $\emin(\omega \cdot n)$-instances will, at a finite stage, stabilize to their actual least element. It is now easy to see that if the least elements are different, and thus have different rank, the corresponding $\iUF{1}{n}$-instances will be non-isomorphic (and vice versa: isomorphic if the least elements are the same). Thus, $f$ is a reduction.
    
    ``$\leq_c$": Again, we construct the reduction in stages. For a given $\iUF{1}{n}$-instance $W_e$, we construct a $\emin(\omega^2)$-instance $W_{f(e)}$. At any finite stage, if $\A$ is the algebra presented at this stage of $W_e$, compute $m(\A)$ and $I(\A)$, which characterize the isomorphism type. Then, put the element of rank $\omega \cdot m(\A) + I(\A)$ in $W_{f(e)}$. By Observation \ref{obs:uf_graph_decreasing}, $\omega \cdot m(\A) + I(\A)$ will be the least element of $W_{f(e)}$ at that stage. 
    
    For any $\iUF{1}{n}$-instance $W_e$, the invariant $(m(\A), I(\A))$ at each stage must stabilize as it is a descending sequence contained in $\omega^2$ by Observation \ref{obs:uf_graph_decreasing}. When $(m(\A), I(\A))$ stabilizes, the isomorphism type and $W_{f(e)}$ also stabilize. Thus, for any two $\iUF{1}{n}$-instances $W_e$ and $W_i$, by considering a late enough stage where both isomorphism types stabilize, we see that the $\iUF{1}{n}$-instances are isomorphic if and only if $\min(W_{f(e)}) = \min(W_{f(i)})$.  Computability of $f$ follows from Proposition \ref{uf_graph_computable}. Thus, $f$ is a computable reduction.
    \end{proof}

\begin{corollary}
	$\iUF{1}{n} <_c \iUF{1}{n+1}$.
\end{corollary}
\begin{proof}
    Follows from Proposition \ref{uf_emin} together with Proposition \ref{E_min_strictly_ascending}.
\end{proof}

We also show that $\iUF{1}{\fg}$ is bounded by $\emin(\omega^2)$. The reverse reducibility fails here, because already for the empty set as $\emin(\omega^2)$-instance we have to commit to a certain number of generators in the $\iUF{1}{\fg}$-instance. In contrast to the $n$-generated case, in this case, there is no ``top element" from which any other algebra can be obtained through quotienting.

\begin{proposition}\label{uf_fg_emin}
    $\iUF{1}{\fg} <_c \emin(\omega^2)$.
\end{proposition}
\begin{proof}
    ``$\leq_c$": The proof is identical to the $\leq_c$ direction of Proposition \ref{uf_emin}.
    
    ``$\not\geq_c$": Assume that there exists a computable reduction $f$ from $\emin(\omega^2)$ to $\iUF{1}{\fg}$. In stages we construct two instances $W_e$ and $W_i$ of $\emin(\omega^2)$ that cannot be correctly mapped to $\iUF{1}{\fg}$-instances by $f$. By the recursion theorem, we fix the indices $e,i$. 
    \begin{itemize}
        \item At stage $0$, both $W_e$ and $W_i$ are empty. 
        \item Then, alternate between the following two cases, starting with the first: 
        \begin{itemize}
            \item At this point, $W_e$ and $W_i$ are in the same $\emin(\omega^2)$-class. We wait until we see that the algebras presented by the images under $f$ of $W_e$ and $W_i$ are isomorphic. We can decide this by computing (Proposition \ref{uf_graph_computable}) and comparing their invariants. Once that happens, i.e.\ they have the same invariant $m, I$, we put the element of rank $\omega \cdot m + I$ into $W_e$. If that never happens, $f$ is not a reduction. Observation \ref{obs:uf_graph_decreasing} (together with our construction of previous stages: the least element of $W_i$, if it exists, must have come from the isomorphism type of a pre-quotient) ensures that now $W_e$ and $W_i$ are in different $\emin(\omega^2)$-classes.
        \item At this point, $W_e$ and $W_i$ are in different $\emin(\omega^2)$-classes. We wait until we see that the presentations are non-isomorphic (again, this is decidable by Proposition \ref{uf_graph_computable}). Once that happens, put the current least element of $W_e$ into $W_i$. If that never happens, $f$ is not a reduction.  
        \end{itemize}
            \end{itemize}
            
        We claim that we cannot alternate between the two cases infinitely often. Indeed, every time we switch cases, the invariant $(m, I)$ strictly decreases; but $(m, I)$ must stabilize as it is a strictly decreasing sequence contained in $\omega^2$ by Observation \ref{obs:uf_graph_decreasing}. Thus, we must eventually stabilize to one of the two cases. In either case, we obtain diagonalization, so $f$ cannot be a computable reduction.
\end{proof}

\subsection{\texorpdfstring{$\AG{n}$}{AGn}}\label{ag_section}

We further show that the isomorphism relations of $n$-generated abelian groups also fall in the degrees of $\emin(\omega \cdot n)$ and that the isomorphism relation of finitely generated abelian groups are below $\emin(\omega^2)$. This follows from very well-known results. 

\begin{proposition}\label{ag_emin}
    $\iAG{n} \equiv_c \emin(\omega \cdot n)$.
\end{proposition}

\begin{proof}
    ``$\leq_c$": By the fundamental theorem of finitely generated abelian groups, every algebra in $\AG{n}$ has a unique decomposition into $\mathbb{Z}^i \oplus \mathbb{Z}/k_1 \mathbb{Z} \oplus ... \oplus \mathbb{Z}/k_m \mathbb{Z}$, where $k_i$ divides $k_{i+1}$ (thus also the order is unique). The numbers $i,k_1,...,k_m$ can be computed from a finite presentation, e.g.\ by computing the Smith normal form (see e.g.\ \cite{newman1997smith}) from the equations in the presentation. 

    So we can construct our reduction by computing, at each finite stage, the numbers $i,k_1,...,k_m$ from the finite presentation and add the element $\omega \cdot i + \langle k_1,...,k_m\rangle$ to the $\emin(\omega \cdot n)$ instance. Note that as long as $i$ does not change, each of $k_1,...,k_m$ can only decrease, so by choosing a suitable tuple encoding we can make sure that the element that is last added is the least element of the $\emin(\omega \cdot n)$ instance. By uniqueness of the decomposition, it follows immediately that this reduction is correct (i.e.\ it preserves equivalence and non-equivalence).

    ``$\geq_c$": We can encode the element of rank $\omega \cdot i + k$, with $i < n$, by the abelian group $\mathbb{Z}^i \oplus \mathbb{Z}/k!$. It is easy to check that these are c.e.\ presentable. Also, if the rank decreases, we can obtain the abelian group that encodes it by quotienting. 
\end{proof}

Analogous to Proposition \ref{uf_fg_emin}, the ``$\leq_c$"-direction in the proof of Proposition \ref{ag_emin} also works for finitely many generators, and the other direction can be diagonalized against (in this case using the invariant from Proposition \ref{ag_emin}). So we obtain the following.
\begin{proposition}\label{ag_fg_emin}
   $\iAG{\fg} <_c \emin(\omega^2)$
\end{proposition}

However, note that since for both $\UF{1}{\fg}$ and $\AG{\fg}$ we have these invariants corresponding to ordinals below $\omega^2$, we can at least show that they are equivalent.
\begin{proposition}\label{ag_equiv_uf}
    $\iAG{\fg} \equiv_c \iUF{1}{\fg}$
\end{proposition}
\begin{proof}
    ``$\leq_c$": At each stage, compute the invariant numbers $i, k_1, ..., k_m$ of the $\iAG{\fg}$ instance as described in the ``$\leq_c$" direction of Proposition \ref{ag_emin}. Construct the $\iUF{1}{\fg}$ instance as done in the ``$\geq_c$" direction of Proposition \ref{uf_emin}.

    ``$\geq_c$": At each stage, compute the invariant numbers $m, I, F$ of the $\iUF{1}{\fg}$ instance as described in the ``$\leq_c$" direction of Proposition \ref{uf_emin}. Construct the $\iAG{\fg}$ instance as done in the ``$\geq_c$" direction of Proposition \ref{ag_emin}.
\end{proof}

\section{Below \texorpdfstring{$\eqce$}{=ce} and ACC}\label{sec:below_eqce}
In this section, we study further equivalence relations below $\eqce$. Here, we focus on the consequences of the ascending chain condition (ACC) in congruence lattices, a structural property of certain classes of algebras. This allows us to also derive general results, as a step towards a systematic picture of equivalence relations below $\eqce$. 

But before we explore the connection to ACC, we start this section by giving a lower bound to isomorphism relations on c.e.\ structures in a variety. 

\begin{proposition}\label{least_isomorphism_relation}
    Under computable reducibility, there is a least non-trivial (i.e.\ more than one class) isomorphism relation, namely, the isomorphism relation on $2$-generated sets without function symbols. 
\end{proposition}

\begin{proof}
    In the absence of function symbols (including constants), an algebra is a set without structure, and the term algebra only contains the generators. Call this variety $\mathcal S$. If we let $x_0, x_1$ be the two generators, $\mathcal S_2$ contains exactly two algebras (up to isomorphism), $\{x_0\}$ and $\{x_0,x_1\}$.
    We argue that the corresponding isomorphism relation is computably reducible to any other non-trivial isomorphism relation, say $\C \in \{\V{n}, \V{\fg}\}$ where $V{}$ is a variety. 
    Given a $\mathcal S_2$-instance $e$, as long as $x_0 = x_1$ has not entered $W_e$, we define $W_{f(e)}$ to be the empty set (the trivial congruence relation). As soon as $x_0=x_1$ entered $W_e$, we define $W_{f(e)}$ to be $\nat$ (the full congruence relation), namely, start enumerating every $n\in \nat$ into $W_{f(e)}$. It is immediate that in a non-trivial isomorphism relation, $\emptyset$ and $\nat$ must be in different classes.
\end{proof}

Note that this equivalence relation is not c.e. Therefore, since isomorphism relations on finite presentations are at most c.e.\ (Observation \ref{fp_sigma_1}), they can never be computably equivalent to any isomorphism relation on c.e.\ algebras, except in the trivial degree.

\subsection{ACC and homomorphic openness}

Here, we give a formal definition of the ACC on congruence lattices. The \emph{congruence lattice} of an algebra is the lattice of all congruence relations of the algebra, ordered by inclusion.

\begin{definition}
	A partially ordered set $(P,\leq)$ (i.e. $\leq$ is a reflexive, antisymmetric, and transitive binary relation on $P$) satisfies the \emph{ascending chain condition (ACC)} if for every ascending sequence
	\[
	p_0 \leq p_1 \leq p_2 \leq ...
	\]
	of elements $p_i \in P$, there exists a $n \in \nat$ s.t.
	\[
	p_n = p_{n+1} = p_{n+2} = ...
	\]

    We say an algebra has \emph{ACC} if its congruence lattice has ACC. We say a class of algebras has \emph{ACC} if all algebras in the class have ACC. 
\end{definition}

Note that for a variety $\V{}$, the class $\V{n}$ has ACC iff $\V{}[\emptyset]$ (the free algebra on $n$ generators) has ACC, and $\V{\fg}$ has ACC iff for every $n\geq 1$, $\V{\fg}[n,\emptyset]$ has ACC. Furthermore, if $\V{\fg}$ has ACC, so does $\V{n}$ for every $n \geq 1$.

Also, we immediately get the following very simple but important lemma.
\begin{lemma}\label{acc_isom_type}
		Let $\V{}$ be a variety and let $\C \in \{\V{n}, \V{\fg}\}$ have ACC. Then, for each $e$ there exists a stage $s$ of the enumeration of the equations at which the isomorphism type stabilizes, i.e. 
        \[
        \C[W_{e,s}] \cong \C[W_{e,s+1}] \cong \C[W_{e,s+2}] \cong ...
        \]
	\end{lemma}
	\begin{proof}
		Infinitely many changes of the isomorphism type would require infinitely many changes of the congruence relation, contradicting ACC. 
	\end{proof}

\begin{definition}
    A class of algebras is \emph{homomorphically open} if it contains two algebras $\A,\B$, s.t.\ $\A \twoheadrightarrow \B$, $\B \twoheadrightarrow \A$, but $\A \not\cong \B$.
\end{definition}
So, being homomorphically open means containing two algebras that are homomorphic images (i.e.\ quotients) of each other, but are not isomorphic. Note that this property is related to non-Hopfianity --- we say an algebra $\A$ is \emph{Hopfian} if it is not isomorphic to any of its proper quotients. Below, we collect some important observations that connect the properties that we mentioned.

\begin{observation} [Folklore]\label{acc_hom_open_facts}
Let $\V{}$ be a variety and let $\C \in \{\V{n}, \V{\fg}\}$. Then,
        \begin{enumerate}
    		\item $\C$ has ACC $\iff$ every algebra in $\C$ has a finite presentation,
	   	\item $\C$ has ACC $\implies$ every algebra in $\C$ is Hopfian,
		  \item every algebra in $\C$ is Hopfian $\implies$ $\C$ is not homomorphically open.
        \end{enumerate}
\end{observation}
\begin{proof}
    \begin{enumerate}
        \item For the $\implies$-direction, assume ACC and take any algebra $\A \in \C$ and fix a presentation of $\A$. Let $R$ be the set of equations in this presentation. Define
        \begin{itemize}
            \item $R_0 := \emptyset$
            \item If there is an equation from $R$ that is not yet implied by $R_i$, define 
            $R_{i+1} := R_i \cup \{s \approx t\}$. Otherwise, end the construction.
        \end{itemize}
        If the construction never stops, the congruences generated by the $R_i, i \geq 0$ would strictly increase and contradict ACC. Thus, there is some finite set $R_i$ of equations that already implies all equations in $R$. This proves the existence of a finite presentation for $\A$.
        
        For the $\impliedby$-direction, take any infinite chain of congruences and observe that its union must have a finite presentation. But then there is a congruence in this chain that contains this finite presentation, so the chain stabilizes.
        \item Assume that there exists a non-Hopfian algebra $\A$, i.e.\ there is a non-injective homomorphic mapping witnessing $\A \twoheadrightarrow \A$. Iterating this homomorphic mapping, we obtain a sequence
        $$\A \twoheadrightarrow \A \twoheadrightarrow \A \twoheadrightarrow \A \twoheadrightarrow \dots$$
        where every surjection is strict. By Lemma \ref{quotient_superset}, this gives an infinitely ascending chain of congruences, a contradiction to ACC.
        \item Let $\A,\B \in \C$ witness $\C$ being homomorphically open. Then $\A$ is a non-trivial homomorphic image of itself via $\A \twoheadrightarrow \B \twoheadrightarrow \A$, therefore non-Hopfian.
    \end{enumerate}
\end{proof}

In particular, if $\A$ is an algebra in a variety $\C$ that has ACC, then any strict quotient of $\A$ is non-isomorphic to $\A$. We will use this fact without explicitly mentioning it in the later sections.

\subsection{Examples of classes with ACC}\label{examples_acc}
We now look at concrete classes that have ACC. As first examples, we consider the varieties $\mathcal{CS}$ and $\mathcal{CM}$ (commutative semigroups and commutative monoids). In these varieties, every finitely generated algebra has a finite presentation \cite{redei2014theory} and Observation \ref{acc_hom_open_facts} shows that this is equivalent to having ACC. From the proofs of Proposition \ref{uf_emin} and Proposition \ref{ag_emin}, it is easy to extract that the classes of finitely generated algebras of $\UF{1}{}$ and $\AG{}$ also have ACC.

So we have $\V{n},\V{\fg}$, for $\V{} \in \{\UF{1}{},\AG{},\CS{}, \CM{}\}$, as examples of classes with ACC.
Let us look at how these compare to each other. We will show that already $\iCS{2}$ (and thus also $\iCM{2}$) is above $\emin(\omega^2)$ and thus above any $\iUF{1}{n}$ and $\iAG{n}$.

\begin{proposition}\label{emin_below_cs2}
	$\emin(\omega^2) \leq_c \iCS{2}$.
\end{proposition} 
\begin{proof}
	We build a computable reduction $f$ by constructing a $\iCS{2}$-instance $W_{f(e)}$ (with generators $x_0,x_1$) from a $\emin(\omega^2)$-instance $W_e$ in stages as follows:
    \begin{itemize}
        \item As long as $W_e$ stays empty, leave $W_{f(e)}$ empty as well.
		\item If an element enters $W_e$, s.t.\ now the least element has rank $\omega \cdot i + j$, we do the following:
		\begin{itemize}
			\item We put $x_0^j x_1^i = x_0^{j+1} x_1^{i}$ in $W_{f(e)}$ to encode the current situation.
			\item We want to ensure that no information from previous stages remains. So we also add $x_0 x_1^k = x_0^2 x_1^k$, for $k > i$.
		\end{itemize}
	\end{itemize}

    It remains to be verified that this is a reduction, i.e. 
    \[
	e ~\emin(\omega^2)~ i \iff \CS{2}[f(e)] \cong \CS{2}[f(i)].
	\]
    ``$\implies$": Observe that for any $\emin(\omega^2)$-instance there is a stage at which it stabilizes to the actual least element of rank $\omega \cdot i + j$. Two equivalent $\emin(\omega^2)$-instances have the same least element, i.e. the same $i,j$, so after both have stabilized, the ``coding equation" $x_0^{j+1} x_1^i = x_0^{j+2} x_1^{i}$ and the ``deletion equations" $x_0 x_1^k = x_0^2 x_1^k$ for $k > i$ are added to both images under $f$ (the $\iCS{2}$-instances). Now, it is easy to check that by the compatibility of congruence relations with function application, these two equations always imply every equation added at a previous stage. Therefore, both $\iCS{2}$-instances have the same congruence relations and are thus isomorphic.

    ``$\impliedby$": Assume that two $\emin(\omega^2)$-instances $S_1, S_2$ are not equivalent. So after stabilizing, the least element of $S_1$ with rank $m_1 := \omega \cdot i_1 + j_1$ and the least element of $S_2$ with rank $m_2 := \omega \cdot i_2 + j_2$ are different. W.l.o.g.\ assume $m_1 < m_2$. Let $\A_1, \A_2$ be the presentations that are the images of $S_1$ and $S_2$, respectively, under $f$. It is straightforward to check that, by our construction, the last two equations (the ``coding equation" and the ``deletion equation") that are added to $\A_1$ imply any equation added to $\A_2$. So $\A_1$ is a quotient of $\A_2$. Furthermore, it is a strict quotient, since the last coding equation entering $\A_1$ cannot be a consequence in $\A_2$. Due to ACC, strict quotienting must change the isomorphism type, so $\A_1$ and $\A_2$ are, in both cases, not isomorphic.
\end{proof}

The same idea works for a generalized version.
\begin{theorem}
        $\emin(\omega^n) \leq_c \iCS{n}$, for $n \geq 2$.
\end{theorem}
\begin{proof}
        We generalize the previous proof to encode $n$ coefficients. We use the equation $x_0^{c_0}x_1^{c_1}...x_{n-1}^{c_{n-1}} = x_0^{c_0 + 1}x_1^{c_1}...x_{n-1}^{c_{n-1}}$ to encode that $\omega^{n-1}c_{n-1} + ... + \omega c_1 + c_0$ is the rank of the least element. For erasing information from previous stages we add the equations $x_0x_1^{k_1}...x_{n-1}^{k_{n-1}} = x_0^{2}x_1^{k_1}...x_{n-1}^{k_{n-1}}$ for all $k_1,...,k_{n-1}$ for which $\omega^{n-1}k_{n-1} + ... + \omega k_1 > \omega^{n-1}c_{n-1} + ... + \omega c_1$. Verification of this being a reduction works analogously.
\end{proof}

Next, we show that $\emin(\omega^n)$ is optimal, i.e.\ $\emin(\omega^n+1)$ does not reduce to $\iCM{n}$ and thus also not to $\iCS{n}$. This requires a bit of setup. 

In the following, let $\mathbb{N}_\infty := \mathbb{N} \cup \{\infty\}$. The natural order on $\mathbb{N}_\infty$ induces a partial order $\leq$ on $n$-tuples of elements from $\mathbb{N}_\infty$. For a $v \in \mathbb{N}_\infty^n$ and a set $S \subseteq \mathbb{N}^n$, we say \emph{$v$ beats $S$} if $v \not\geq w$ for all $w \in S$.

Let $F_n = \{ x_0^{\epsilon_0} ... x_{n-1}^{\epsilon_{n-1}}\}$ be the free commutative monoid, where we consider $x_0^{\epsilon_0} ... x_{n-1}^{\epsilon_{n-1}}$ as the canonical representative of each element, and let $\le_{\lex}$ be the natural lexicographic (well-)order on $F_n$. For $i < n$, we define $\pi_i: F_n \to \nat$ by $\pi_i(x_0^{\epsilon_0} ... x_{n-1}^{\epsilon_{n-1}}) = \epsilon_i$.

    We first give a (partial) invariant for every presentation $\A = \langle X_n \mid W\rangle$ in $\CM{n}$ that is nontrivial (i.e.~$\A\not\cong F_n$). 
    Let $\varphi: F_n \to \A$ be the natural projection map and define $$S(\A) := \{ (\epsilon_0,...,\epsilon_{n-1})\in \mathbb{N}^n \mid (x_0^{\epsilon_0} ... x_{n-1}^{\epsilon_{n-1}}) >_{\lex} \min(\varphi^{-1}(x_0^{\epsilon_0} ... x_{n-1}^{\epsilon_{n-1}}))\}.$$ Note that since $\A$ is nontrivial, $S(\A)$ is nonempty. Intuitively, we represent each equivalence class by its least member, so $S(\A)$ is the collection of (exponents of) words that no longer represent any elements of $\A$.

    It is easy to see that the words that $S(\A)$ represents are closed under multiplying by any generator $x_i$, $i<n$. Namely, if $({\epsilon_0},... ,{\epsilon_{n-1}}) \in S(\A)$, then for any $i<n$, $({\epsilon_0},..., \epsilon_i+1,..., {\epsilon_{n-1}})$ is also in $S(\A)$.
    We can equivalently say that $S(\A)$ is closed upwards under the natural partial order.

Now, for every set $M \subseteq \mathbb{N}^n$ we can define an invariant $\gamma(M) := \omega^{n-1} a_{n-1} + ... + \omega a_1 + a_0$ by
\begin{align*}
&A_i := \{x \in \mathbb{N}_\infty^n \mid \text{$x = (x_0,...,x_{n-1})$ has $i$-many $\infty$, $x$ beats $M$, }\\ &\text{and no $x^\prime := (x^\prime_0,...,x^\prime_{n-1}) \neq x$ with $x^\prime_j = x_j$ or $x^\prime_j = \infty$ for $j<n$ beats $M$}\}\\
&a_i := |A_i|.
\end{align*}

\begin{figure}
\begin{center}
    \begin{tikzpicture}[scale=1]

  \draw[step=0.5cm, gray!40, very thin] (0,0) grid (6,6);

  \draw[->, very thick] (-0.2,-0.2) -- (6.5,-0.2) node[right] {$x_0$};
  \draw[->, very thick] (-0.2,-0.2) -- (-0.2,6.5) node[above] {$x_1$};

  \coordinate (A) at (1.5,5);
  \coordinate (B) at (4,2);
  \coordinate (C) at (2,4);

\fill[gray!40, opacity=0.6]
    (A) -- (2,5) -- (C) -- (4,4) --
    (B) -- (7.5,2) -- (7.5,7.5) -- (1.5,7.5) --
    cycle;

      \node at (4.5,5) {$S(\A)$};

    \fill (1.5,4.5) circle (1pt);
    \fill (1.5,4) circle (1pt);
    \fill (1.5,3.5) circle (1pt);
    \fill (1.5,3) circle (1pt);
    \fill (1.5,3) circle (1pt);
    \fill (1.5,2.5) circle (1pt);
    \fill (1.5,2) circle (1pt);

    \fill (2,3.5) circle (1pt);
    \fill (2,3) circle (1pt);
    \fill (2,2.5) circle (1pt);
    \fill (2,2) circle (1pt);

    \fill (2.5,3.5) circle (1pt);
    \fill (2.5,3) circle (1pt);
    \fill (2.5,2.5) circle (1pt);
    \fill (2.5,2) circle (1pt);

    \fill (3,3.5) circle (1pt);
    \fill (3,3) circle (1pt);
    \fill (3,2.5) circle (1pt);
    \fill (3,2) circle (1pt);

    \fill (3.5,3.5) circle (1pt);
    \fill (3.5,3) circle (1pt);
    \fill (3.5,2.5) circle (1pt);
    \fill (3.5,2) circle (1pt);

\draw (0,0) -- (0,7.5);
\draw (0.5,0) -- (0.5,7.5);
\draw (1,0) -- (1,7.5);

\draw (0,0) -- (7.5,0);
\draw (0,0.5) -- (7.5,0.5);
\draw (0,1) -- (7.5,1);
\draw (0,1.5) -- (7.5,1.5);

\end{tikzpicture}
\end{center}
\caption{A visualization of how the invariant is calculated for a particular $S(\A)$ of a $2$-generated presentation $\A$. The dots represent elements in $A_0$ and the lines represent elements in $A_1$. From their cardinalities, we get $\gamma(A) = \omega\cdot 7 + 22$.}
\label{fig:invariant_example}
\end{figure}
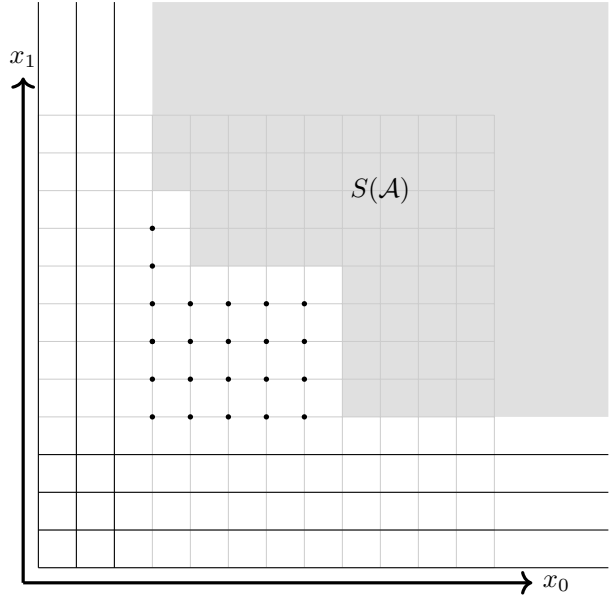

From this we can define an invariant $\gamma(\A)$ for a non-trivial presentation $\A$ by $\gamma(\A) := \gamma(S(\A))$. Figure \ref{fig:invariant_example} visualizes this invariant for an example of a $2$-generated presentation. 

For this definition to be useful to us, we must ensure that the coefficients are natural numbers. So we need the corresponding sets to be finite.

\begin{observation}\label{CMn_invariant finite}
    In the above definition of $\gamma(M)$, for all $i<n$, $A_i$ is finite.
\end{observation}
\begin{proof}

    Towards a contradiction, assume $A_i$ is infinite. By the chain-antichain principle, $A_i$ must contain either an infinite antichain or an infinite chain. The former we can exclude due to Dickson's Lemma (as it is easy to see that Dickson's Lemma also holds for $\mathbb{N}_\infty^n$). In the latter case, take an element $x = (x_0,...,x_{n-1})$ in the chain s.t.\ for all $i<n$ either $x_i$ is maximal for coordinate $i$ in the chain or no maximal element exists at coordinate $i$ in the chain. Now consider the supremum $x^\prime = (x^\prime_0,...,x^\prime_{n-1})$ of the infinite chain. Notice that $x^\prime \neq x$, and that $x^\prime_j = x_j$ or $x^\prime_j = \infty$ for all $i<n$. Furthermore, it is easy to verify that also $x^\prime$ beats $M$, a contradiction to $x$ being in $A_i$.
\end{proof}

\begin{observation}\label{CMn_invariant_strictly_decreases}
    If $M \supseteq N$ then $\gamma(M) \leq \gamma(N)$. Furthermore, if $M,N$ are closed upwards (under the natural partial order) then $M \supset N$ implies $\gamma(M) < \gamma(N)$.
\end{observation}
\begin{proof}
Assume that $M \supseteq N$. For $i<n$, let $A_i$ be the sets from which $\gamma(M) = \omega^{n-1}a_{n-1} + ... + a_0$ is calculated and let $B_i$ be the sets from which $\gamma(N) = \omega^{n-1}b_{n-1} + ... + b_0$ is calculated. We claim that for each $i<n$ and for each element in $A_i \setminus B_i$ there is an element in $B_j \setminus A_j$ for a $j>i$. Then, in particular, we have $A_{n-1} \subseteq B_{n-1}$. This implies that for the largest $k$ with $A_k \neq B_k$ it must be that $A_{k} \subsetneq B_{k}$. Furthermore, it is easy to check that under the assumption that $M,N$ are closed upwards and $M \supset N$, such a $k$ with $A_k \neq B_k$ must exist.

To prove the claim, let $i<n$ and let $z \in A_i \setminus B_i$. Observe that since $z$ beats $M$, it must also beat $N$. So for $z$ not to be in $B_i$ there must be a $z^\prime = (z^\prime_0,...,z^\prime_{n-1}) \neq z$ with $z^\prime_k = z_k$ or $z^\prime_k = \infty$ for $k<n$ that beats $N$. Pick such a $z^\prime$ that has a maximal number $j$ of coordinates with value $\infty$. It is easy to check that this $z^\prime$ is in $B_j$. However, $z^\prime$ does not beat $M$ (otherwise $z$ would not be in $A_i$) and therefore $z^\prime \in B_j \setminus A_j$.
\end{proof}

In summary, $\gamma$ is a strictly decreasing $\nat^n$ invariant for nontrivial presentations in $\CM{n}$. However, note that $\gamma$ may be different for two presentations of the same isomorphic structure (like e.g.\ for the presentations in Example \ref{ex:non_trivial_isomorphism}). Also, it is easy to see that two different presentations may have the same $\gamma$ and so even for presentations it is not a complete invariant.

    In the following, we observe that we can computably approximate $\gamma$ ``from above".

\begin{observation}\label{CMn_invariant_approximate}
    There is a computable function that, given a non-trivial presentation $\A = \langle X_n \mid W_e\rangle$ and $s \in \omega$, outputs a $\gamma_s(\A)$, s.t.\
    \begin{enumerate}
        \item $\gamma(\A) \leq \gamma_s(\A)$ for all $s \in \omega$,
        \item $\gamma_t(\A) \leq \gamma_s(\A)$ for all $t\geq s$, 
        \item there is a stage $s$ s.t.\ $\gamma_t(\A) = \gamma(\A)$ for all $t \geq s$.
    \end{enumerate}
\end{observation}
\begin{proof}
    Notice that, given a presentation $\A$, $S(\A)$ can be computably enumerated from an enumeration of the generated congruence relation. By $S(\A)_s$ we denote $S(\A)$ enumerated up to stage $s$. Since we only consider non-trivial presentations, we can assume w.l.o.g.\ that $S(\A)_s$ is non-empty. 
    
    Given a presentation $\A$ and $s \in \omega$, we define $\gamma_s(\A) := \gamma(S(\A)_s)$. We claim that we can compute $\gamma_s$. Let $z = (z_0,...,z_{n-1})$ be the coordinate-wise maximum of all elements in the finite set $S(\A)_s$. Let $A_i$, $i<n$, be as in the definition of the invariant, i.e.\ the sets that determine the invariant coefficients. By Observation \ref{CMn_invariant finite}, these are finite sets. Their cardinalities are the coefficients of $\gamma_s$ so it is sufficient to show that we can compute $A_i$, for all $i<n$. We do this inductively by assuming we already computed $A_{j}$ with $j>i$.
    To compute $A_{i}$, go through all $n$-tuples $x = (x_0,...,x_{n-1})$ with $i$-many infinities and $x_k \leq z_k$ for all non-infinity coordinates $k$. Pick an $x$ if it fulfills the requirements of the definition of $A_i$ (this requires having computed all $A_j$, for $j>i$). First, it is easy to see that we have a finite search space, so the computation terminates. Second, any tuple outside the search space cannot be in $A_i$: if an $x = (x_0,...,x_{n-1}) \in A_i$ has $x_k > z_k$, then $(x_0,...,x_{k-1},\infty, x_{k+1},...,x_{n-1})$ would still beat $S(\A)_s$, a contradiction.

    Conditions 1.\ and 2.\ then follow from Observation \ref{CMn_invariant_strictly_decreases}. To verify condition 3., notice that since $S(\A)$ is closed upwards under the natural partial order, it is ``generated by" its minimal elements. By Dickson's Lemma there can only be finitely many minimal elements and thus there is a stage $s$ at which all of them have been enumerated. Since it is easy to observe that any set of tuples has the same invariant as its upwards closure, $\gamma_s(\A) = \gamma(\A)$.
\end{proof}

\begin{theorem}\label{iCM-non-reduction}
    $\emin(\omega^n+1) \not\leq_c \iCM{n}$.
\end{theorem}
\begin{proof}
    Now suppose $f$ is a computable reduction from $\emin(\omega^n+1)$ to $\iCM{n}$. By the recursion theorem, fix two indices $e$ and $i$ that we control. We consider the following algorithm enumerating the instances $W_e$ and $W_i$ of $\emin(\omega^n+1)$:
    \begin{itemize}
        \item Stage 0. We start with $W_e$ empty and $W_i = \{\omega^n\}$. We wait until at least one of $W_{f(e)}$ or $W_{f(i)}$ is nonempty. Let $s_0$ be the stage in the enumeration of $W_{f(e)}$ and $W_{f(i)}$ at which this happens. We then make $W_e = W_i = \{\omega^n\}$ and proceed to the next stage.

        \item Stage $2\ell-1$. We wait until we see that (after stage $s_{2\ell -2}$) $W_{f(e)}$ and $W_{f(i)}$ are $\iCM{n}$-equivalent. This we can decide by Theorem \ref{isomorphism_problem_decidable}. Let $s_{2\ell -1}$ be the stage in the enumeration of $W_{f(e)}$ and $W_{f(i)}$ at which this happens. 
        Going through $t\geq 0$, we compute the approximation $\gamma_t(\langle X_n \mid W_{f(i), s_{2\ell-1}}\rangle)$ according to Observation \ref{CMn_invariant_approximate} and put it into $W_i$. Once the least element of $W_i$ changes, we stop and proceed to the next stage.

        \item Stage $2\ell$. We wait until we see that (after stage $s_{2\ell -1}$) $W_{f(e)}$ and $W_{f(i)}$ are $\iCM{n}$-non-equivalent. Again, this is decidable. Let $s_{2\ell}$ be the stage in the enumeration of $W_{f(e)}$ and $W_{f(i)}$ at which this happens. 
        We then copy everything in $W_i$ into $W_e$ and proceed to the next stage.
        
    \end{itemize}

    Our construction ensures that at the beginning of the even stages, $W_e\subsetneq W_i$ and they are not $\emin(\omega^n+1)$-equivalent, but $W_{f(e)}$ and $W_{f(i)}$ are $\iCM{n}$-equivalent; and at the beginning of the odd stages, $W_e = W_i$ and, in particular, are $\emin(\omega^n+1)$-equivalent, but $W_{f(e)}$ and $W_{f(i)}$ are not $\iCM{n}$-equivalent. 
    We further claim that at odd stags, if we see $W_{f(e)}$ and $W_{f(i)}$ becoming $\iCM{n}$-equivalent then the construction will always eventually lower the least element of $W_i$ and thus proceed to the next stage.

    Indeed, at every stage $s_{2\ell+1}$ of the enumeration of $W_{f(e)}$ and $W_{f(i)}$, for $\ell > 0$, we consider how $W_{f(e)}$ and $W_{f(i)}$ change since stage $s_{2\ell-1}$. They were $\iCM{n}$-equivalent at stage $s_{2\ell-1}$ but non-equivalent at stage $s_{2\ell}$. Thus, there must be some \emph{new} (i.e.~not implied by the previous identities) identity entering one of the two sets. However, at stage $s_{2\ell+1}$, $W_{f(i)}$ and $W_{f(e)}$ are equivalent again; but by ACC and non-Hopfianity, we cannot have one of them not changing $\iCM{n}$-class. Thus, both of them must have enumerated some new identity. In particular, this means $\langle X_n \mid W_{f(i),s_{2\ell+1}}\rangle$ must be a proper quotient of $\langle X_n \mid W_{f(i),s_{2\ell-1}}\rangle$. Thus, by Observation \ref{CMn_invariant_strictly_decreases},
    $$\gamma(\langle X_n \mid W_{f(i),s_{2\ell+1}}\rangle) < \gamma(\langle X_n \mid W_{f(i),s_{2\ell-1}}\rangle).$$
    The least element of $W_i$ at the beginning of an odd stage $2\ell+1$ was added in stage $2\ell-1$. It is an approximation $\gamma_u(\langle X_n \mid W_{f(i),s_{2\ell-1}}\rangle)$, for some $u>0$, and thus, by Observation \ref{CMn_invariant_approximate}, is above or equal to $\gamma(\langle X_n \mid W_{f(i),s_{2\ell-1}}\rangle)$. So, again by Observation \ref{CMn_invariant_approximate}, there must be a $t$, s.t.\ $\gamma_t(\langle X_n \mid W_{f(i),s_{2\ell+1}}\rangle)$ is strictly below that least element. This proves the claim.
    
    But now, by ACC, $W_{f(i)}$ and $W_{f(e)}$ must eventually stop changing $\iCM{n}$-classes; thus, we will stop at a finite stage. In either odd or even stages, we achieve diagonalization, so $f$ cannot be a reduction. 
\end{proof}

\begin{corollary}
    $\emin(\omega^\omega) \not\leq_c \iCM{\fg}$.
\end{corollary}
\begin{proof}
    We diagonalize against a computable reduction $f$ from $\emin(\omega^\omega)$ to $\iCM{\fg}$. Fix again two indices $e$ and $i$. Let $\langle k,j\rangle$ (which is code for $\langle X_k \mid W_j\rangle_{\CM{}}$) be the image of $i$ under $f$. Do the same construction, only at stage $0$ put $\omega^k$ into $W_i$ (and later into $W_e$).
\end{proof}

Thus, we obtain a picture similar to the results below $\omega^2$ that we had in Section \ref{uf_section} and Section \ref{ag_section}. Only we could not establish reducibility of the isomorphism relations to the $\emin$ relations. The difficulty is that for commutative semigroups, we do not know of any useful (complete) invariant for isomorphism types. In particular, the known proof that their isomorphism problem (i.e.\ for finite presentations) is decidable (see Theorem \ref{isomorphism_problem_decidable}) is not uniform, i.e.\ the instance (a pair of presentations) of the problem is reduced as a whole to another (decidable) problem, and there is no known reduction that uniformly maps presentations. So even a positive answer to the following, rather innocent, question might require a better understanding of the isomorphism problem for commutative semigroups. 

\begin{question}
    Is there a computable ordinal $\alpha$, s.t.\ $\emin(\alpha)$ is above $\iCS{n}$, for $n \geq 2$ (or above $\iCS{\fg}$)?
\end{question}

\subsection{Implications of ACC}
In the remaining part of Section \ref{sec:below_eqce}, we show some general results concerning isomorphism relations below $\eqce$. It turns out that ACC has a lot of implications, most importantly being below $\eqce$ (see Theorem \ref{acc_upper_bound}). These general results also make the picture of our example varieties $\AG{}, \UF{1}{}, \CS{}, \CM{}$ more complete (see Figure \ref{fig:below=ce}).

First, we generalize a result that we had for $\UF{1}{}$ and $\AG{}$, namely the $\not\geq_c$-direction of Proposition \ref{uf_fg_emin} and Proposition \ref{ag_fg_emin}. Recall that $\IP(\V{})$ denotes the isomorphism problem for finitely presented algebras in $\V{}$.
\begin{theorem}
    Let $\V{}$ be a variety s.t.\ $\V{\fg}$ has ACC and $\IP(\V{})$ is decidable. Let $(\alpha_n)_{n \geq 1}$ be a sequence of ordinals s.t.\ for all $n \geq 1$, $\iV{n} \leq_c \emin(\alpha_n)$. Then $\emin(\lambda) \not\leq_c \iV{\fg}$, where $\lambda$ is the supremum of $(\alpha_n)_{n \geq 1}$.
\end{theorem}
\begin{proof}
    Assume that there exists a computable reduction $f$ from $\emin(\lambda)$ to $\iV{\fg}$. In stages, we construct two instances $W_i$ and $W_j$ of $\emin(\lambda)$ that cannot be correctly mapped to $\iV{\fg}$-instances by $f$. By the recursion theorem, we fix the indices $i,j$. They are mapped by $f$ to presentations $\langle X_{k_i} \mid W_{e_i}\rangle$ and $\langle X_{k_j} \mid W_{e_j}\rangle$ respectively. Non-uniformly fix $k > k_i$ and let $g$ be the reduction witnessing $\iV{k} \leq_c \emin(\lambda)$. By Observation \ref{emin_red_well_def}, we can assume w.l.o.g.\ that $g$ is well-defined on c.e.\ sets. Furthermore, let $h$ be a function that turns a $k_i$-generated finite presentation into a $k$-generated presentation, i.e.\ for any index $e$,
    \[
    \langle X_{k_i} \mid W_e \rangle \cong \langle X_k \mid W_{h(e)} \rangle.
    \]
    Observe that $h$ is computable: in the $k$-generated presentation, the extra generators $x_{k-1},...,x_{k_i}$ can be collapsed with $x_0$ to get rid of them.
    
    Now, the construction of $W_i, W_j$:
    \begin{itemize}
        \item We start with both $W_i$ and $W_j$ empty. 
        \item Stage $2\ell$. At this stage, $W_i$ and $W_j$ are in the same $\emin(\lambda)$-class. We wait for some stage $s_{2\ell}$ after $s_{2\ell -1}$ (if $\ell >0$) s.t.\ $\langle X_{k_i} \mid W_{e_i, s_{2\ell}}\rangle$ and $\langle X_{k_j} \mid W_{e_j,s_{2\ell}} \rangle$ are isomorphic. We can decide this by our assumption that $\IP(\V{})$ is decidable. And this has to eventually happen, because we assume that $f$ is a reduction and due to ACC the isomorphism types must stabilize at a finite stage (Lemma \ref{acc_isom_type}). So once we found this stage, start enumerating into $W_i$ all elements in $W_{g \circ h (e)}$, where $e$ is an index of $W_{e_i,s_{2\ell}}$. 
        We claim that this makes $W_i$ and $W_j$ non-$\emin(\lambda)$-equivalent.

        \item Stage $2\ell + 1$. At this stage, $W_i$ and $W_j$ are in different $\emin(\lambda)$-classes. We wait until we see that the presentations are non-isomorphic (again, this is decidable by our assumption that $\IP(\V{})$ is decidable). We call this stage $s_{2\ell +1}$. Once that happens (if that never happens, $f$ is not a reduction), start enumerating into $W_j$ all elements in $W_{g \circ h (e)}$, where $e$ is an index of $W_{e_i,s_{2\ell}}$. We claim that this makes $W_i$ and $W_j$ $\emin(\lambda)$-equivalent.
            \end{itemize}

        But due to Lemma \ref{acc_isom_type}, the isomorphism types of the two presentations must eventually stabilize, either at an even or an odd stage. In both cases we obtain diagonalization, as $f$ does not correctly map $W_i$ and $W_j$. 

        The only thing left to prove is the claim that after even stages, $W_i$ and $W_j$ are not $\emin(\lambda)$-equivalent, and after odd stages they are $\emin(\lambda)$-equivalent. We do this inductively.
        \begin{itemize}
            \item For stage $0$, i.e.\ $W_j$ is still empty, note that we chose $k$ to be strictly above $k_i$. Thus, the presentation $\langle k \mid h(W_i)\rangle$ must be a non-isomorphic quotient of $\langle k \mid \emptyset \rangle$. Thus, by monotonicity (Lemma \ref{monotonicity_lemma}) of $g$, $W_i$ will not be enumerated towards the empty set anymore.
            \item For any even stage $2\ell$, $\ell > 0$, first notice that since $h$ preserves the isomorphism type and the reduction $g$ is well-defined on c.e.\ sets and thus monotone (Lemma \ref{monotonicity_lemma}), we know that the images under $g \circ h$ of $\langle X_{k_i} \mid W_{e_i,s}\rangle$, for increasing stages $s$, will have decreasing least elements - strictly decreasing if the isomorphism type changes. By induction hypothesis, our claim is true for all stages before $2\ell$. Thus, our construction is valid up to stage $2\ell$ and ensures that between stages $s_{2(\ell -1)}$ and $s_{2\ell}$, the presentations (the images of $W_i, W_j$ under $f$) become non-isomorphic and then isomorphic again. Since ACC implies non-Hopfianity (Observation \ref{acc_hom_open_facts}), therefore, $\langle X_{k_i} \mid W_{e_i, s_{2(\ell-1)}}\rangle \not \cong \langle X_{k_i} \mid W_{e_i,s_{2\ell}}\rangle$. Now, after stage $2\ell$, $W_i$ is enumerated towards the image under $g \circ h$ of the latter, while $W_j$ is still enumerated towards the image under $g \circ h$ of the former. Thus, they are not $\emin(\lambda)$-equivalent.
            \item For any odd stage $2\ell + 1$, just observe that both are enumerated towards the image under $g \circ h$ of $W_{e_i, s_{2\ell}}$.
        \end{itemize}
        In the proof of the claim, to improve readability, we neglected the fact that $W_i, W_j$ may also contain elements enumerated at previous stages. But since, as observed in the proof of the claim, the images under $g \circ h$ of $\langle X_{k_i} \mid W_{e_i,s}\rangle$ for increasing stages $s$, will have decreasing least elements. So, the elements from previous stages are not relevant. 

        This completes the proof of the theorem. 
\end{proof}

We already know from Proposition \ref{uf_emin} and Proposition \ref{ag_emin} that, for $n \geq 1$, the $\iUF{1}{n}$ and $\iAG{n}$ both form a strictly ascending sequence under $c$-reducibility. We now show that this holds for any variety with ACC in its finitely generated classes and a decidable isomorphism problem (in particular also $\CS{}$ and $\CM{}$). 
\begin{theorem}\label{acc_implies_strictly_increasing}
Let $\V{}$ be variety s.t.\ $\V{\fg}$ has ACC and $\IP(\V{})$ is decidable. Then $\iV{n} <_c \iV{n+1}$. 
\end{theorem}

\begin{proof}
Assume that there exists a computable reduction $f$ from $\iV{n+1}$ to $\iV{n}$. In stages, we construct two instances $W_e$ and $W_i$ of $\iV{n+1}$ that cannot be correctly mapped to $\iV{n}$-instances by $f$. By the recursion theorem, we fix the indices $e,i$. 
    \begin{itemize}
        \item Stage $0$. We start with both $W_e$ and $W_i$ empty. Thus, at this stage, $\V{n+1}[e]$ and $\V{n+1}[i]$ look isomorphic. We look at the stages of enumerations of $W_{f(e)}$ and $W_{f(i)}$ and wait until $\V{n}[f(e)]$ and $\V{n}[f(i)]$ look isomorphic at some stage. This must eventually happen; otherwise, we could stop the construction and $f$ would not be a reduction (remember that due to ACC, by Lemma \ref{acc_isom_type} the isomorphism type stabilizes at a finite stage). Once that happens, define this stage (of the enumerations of $W_{f(e)}$ and $W_{f(i)}$) as stage $t_0$. Then put $\{x_{n-1} = x_{n}\}$ into $W_e$.
        \item Stage $1$. At this stage, $\V{n+1}[e]$ and $\V{n+1}[i]$ look non-isomorphic. We look at the stages of enumerations of $W_{f(e)}$ and $W_{f(i)}$ and wait until $\V{n}[f(e)]$ and $\V{n}[f(i)]$ look non-isomorphic at some stage after $t_0$. Again, this must eventually happen; otherwise $f$ would not be a reduction. Once that happens, define this stage (of the enumerations of $W_{f(e)}$ and $W_{f(i)}$) as stage $t_1$. Then put $\{x_{n-1} = x_{n}\}$ into $W_i$.
        \item Stage $2k$, $k \geq 1$. At this stage, $\V{n+1}[e]$ and $\V{n+1}[i]$ look isomorphic. We look at the stages of enumerations of $W_{f(e)}$ and $W_{f(i)}$ and wait until $\V{n}[f(e)]$ and $\V{n}[f(i)]$ look isomorphic at some stage after $t_{2k-1}$. Again, this must eventually happen; otherwise $f$ would not be a reduction. Once that happens, define this stage (of the enumerations of $W_{f(e)}$ and $W_{f(i)}$) as stage $t_{2k}$. Then put the set of equations that are currently (i.e.\ at stage $t_{2k}$) in $W_{f(e)}$ into $W_e$. Observe that this makes $\V{n+1}[e]$ and $\V{n+1}[i]$ look non-isomorphic: currently (stage $2k$) $\V{n+1}[e]$ looks isomorphic to $\V{n}[f(e)]$ at stage $t_{2k}$, and $\V{n+1}[i]$ looks isomorphic to $\V{n+1}[e]$ after stage $2(k-1)$ which looks isomorphic to $\V{n}[f(e)]$ at stage $t_{2(k-1)}$. And between stages $t_{2(k-1)}$ and $t_{2k}$, $\V{n}[f(e)]$ and $\V{n}[f(i)]$ became non-isomorphic and then isomorphic again. Due to ACC and Observation \ref{acc_hom_open_facts}, we know that quotienting always results in a new isomorphism type and thus $\V{n}[f(e)]$ has at stage $t_{2k}$ a different isomorphism type than at stage $t_{2(k-1)}$.
        \item Stage $2k + 1$, $k \geq 1$. At this stage, $\V{n+1}[e]$ and $\V{n+1}[i]$ look non-isomorphic. We look at the stages of enumerations of $W_{f(e)}$ and $W_{f(i)}$ and wait until $\V{n}[f(e)]$ and $\V{n}[f(i)]$ look non-isomorphic at some stage after $t_{2k}$. Again, this must eventually happen; otherwise, $f$ would not be a reduction. Once that happens, define this stage (of the enumerations of $W_{f(e)}$ and $W_{f(i)}$) as stage $t_{2k+1}$. Then copy all of $W_e$ into $W_i$. Since, by our construction, $W_i$ is always a subset of $W_e$, they will now be equivalent as sets. 
        \end{itemize}
But due to Lemma \ref{acc_isom_type}, the isomorphism types of the two $\V{n}$-presentations must eventually stabilize, either at an even or an odd stage. In both cases we obtain diagonalization, as $f$ does not correctly map $W_i$ and $W_j$. 
\end{proof}

\begin{proposition}\label{prop:cs_cm_chain}
        $\iCS{2} \leq_c \iCM{2} \leq_c \iCS{3} \leq_c \iCM{3} \leq_c \dots$
\end{proposition}

\begin{proof}
    Clearly, $\iCM{n}$ is reducible to $\iCS{n+1}$, since for any $n$-generated commutative monoid there is an isomorphic $(n+1)$-generated commutative semigroup by making the additional generator into an identity element.
    
    Showing $\iCS{n} \leq_c \iCM{n}$ is also straightforward but requires a bit more work to verify.
    To define a computable reduction $f$, for a given $W_e$, encoding a set of equations $R$, let $W_{f(e)}$ also encode the same set $R$. Remember that we cannot just define $f$ as the identity map, because the encoding of equations on terms is different in $\CS{n}$ and $\CM{n}$.
    It remains to be confirmed that isomorphism and non-isomorphism are preserved. Let $\T_\CS{}(X_n)$ be the term algebra over $X_n$ in the signature of commutative semigroups and $\T_\CM{}(X_n)$ the term algebra over $X_n$ in the signature of commutative monoids. Consider the semigroup $\mathcal{S} := \T_\CS{}(X_n)/R_\CS{}$, where $R_\CS{}$ is the congruence relation on $\T_\CS{}(X_n)$ generated by $R$. Thus, $\mathcal{S}$ is presented by $\CS{n}[e]$. Similarly, the analogously defined monoid $\mathcal{M} := \T_\CM{}(X_n)/R_\CM{}$ is presented by $\CM{n}[f(e)]$. We claim that $\varphi: \mathcal{S} \to \mathcal{M}, [t]_{R_\CS{}} \mapsto [t]_{R_\CM{}}$ is a semigroup embedding that is surjective onto $\mathcal{M} \setminus \{[1]_{M}\}$:
    \begin{itemize}
        \item Well-defined: $t_1 R_\CS{} t_2$ implies $t_1 R_\CM{} t_2$, since the derivation of equivalence of $t_1$ and $t_2$ in $R_\CS{}$ is also valid in $R_\CM{}$.
        \item Injective: for $t_1,t_2 \in \T_\CS{}(X_n)$, $t_1 R_\CM{} t_2 \implies t_1 R_\CS{} t_2$. To see this, notice that for a derivation of equivalence of $t_1$ and $t_2$ in $R_\CM{}$, there also exists one that does not use the identity element $1$. This is because, in every equivalence class except the class of $1$, there is a term that does not contain $1$ as a subterm. So we can also derive the equivalence of $t_1$ and $t_2$ in $R_\CS{}$.
        \item Homomorphism and surjective onto $\mathcal{M} \setminus \{[e]_{M}\}$: follows from the definition.
    \end{itemize}
    It is now easy to check that for any two semigroups $\mathcal{S}_1, \mathcal{S}_2$ and their respective images $\mathcal{M}_1, \mathcal{M}_2$ under $f$, an isomorphism between $\mathcal{S}_1,\mathcal{S}_2$ can be composed with $\varphi$ to get an isomorphism between $\mathcal{M}_1,\mathcal{M}_2$, and also vice versa.
\end{proof}

\begin{corollary}
$\iCS{\fg} \equiv_c \iCM{\fg}$.
\end{corollary}

\begin{proof}
    Observe that we have uniformity in the previous Proposition \ref{prop:cs_cm_chain}. 
\end{proof}

The next obvious question is whether the reducibilities in Proposition \ref{prop:cs_cm_chain} are strict. To this, we can only give a partial answer, and it is left open whether or not $\iCS{n} \leq_c \iCM{n}$ is strict.

\begin{proposition}
    $\iCM{n} <_c \iCS{n+1}$, for $n \geq 1$.
\end{proposition}
\begin{proof}
    Note that from $\langle X_{n+1}\mid \emptyset \rangle_\CS{}$ one can obtain, through quotienting, an algebra isomorphic to $\langle X_{n}\mid \emptyset \rangle_\CM{}$ (by turning the extra generator into the identity element). Therefore, one can show non-reducibility by the same idea that is used for showing non-reducibility in Theorem \ref{acc_implies_strictly_increasing}. The only modification in the construction is that at stages $0$ and $1$, instead of adding a collapse of the extra generator, equations that make the extra generator the identity element are added.
\end{proof}

\begin{proposition}
    Let $\V{}$ be any variety. Then $\iV{n} \le_c \iV{\fg}$ for any $n$. If $\V{\fg}$ satisfies ACC and $\IP(\V{})$ is decidable, then the inequality is strict.
\end{proposition}

\begin{proof}
    The $\leq_c$-part is taken care of by $f: e \mapsto \langle n,e \rangle$. If $\V{\fg}$ has ACC $\IP(\V{})$ is decidable, then by Theorem \ref{acc_implies_strictly_increasing} the $\iV{n}$, $n\geq 1$, form a strictly increasing chain. Therefore, since $\iV{\fg}$ is above each of them, it has to be strictly above.
\end{proof}

In the following Theorem, we derive an upper bound of the complexity of $\iV{\fg}$ from ACC alone.

\begin{theorem}\label{acc_upper_bound}
Let $\V{}$ be a variety where $\V{\fg}$ has ACC. Then $\iV{\fg}$ is reducible to $\eqce$.
\end{theorem}

\begin{proof}
    In order to work with stages of enumeration, we introduce the notation $\cong_s$ for the isomorphism relation on presentations enumerated up to stage $s$. So $\cong_s$ is on finite presentations and thus $\Sigma^0_1$ (Observation \ref{fp_sigma_1}). Also (via some suitable encoding), we view the c.e.\ sets on which $\eqce$ is defined as consisting of $\nat$-many columns of $\nat$ and we denote the element $m$ in the $i$-th column by $m^{[i]}$.
    
    Using infinitely many workers, we build a c.e.\ set $W_{f(p)}$ from a given code $p = \langle n,e\rangle$ for a presentation $\langle X_n\mid W_e\rangle_\V{}$. 
	\begin{itemize}
		\item Worker 0: copies $\{m^{[p]}\mid m \in \nat\}$ into $W_{f(p)}$.
		\item Worker $\langle q, r, s \rangle$: waits until $\V{\fg}[q] \cong_s \V{\fg}[r]$. If that happens, the worker copies $W_{f(q),s}$ into $W_{f(r)}$, making sure the copying happens after stage $s$.
	\end{itemize}
	
	``$\implies$": Assume $\V{\fg}[q] \cong \V{\fg}[r]$. Let $x \in W_{f(q),s}$. Let $t$ be the stage at which both isomorphism types stabilize. Worker $\langle q, r, \max(s,t)\rangle$ will eventually see $\V{\fg}[q] \cong_{\max(s,t)} \V{\fg}[r]$ and put $x$ in $W_{f(r)}$.
	
	``$\impliedby$": Assume $\V{\fg}[q] \not\cong \V{\fg}[r]$. By Observation \ref{acc_hom_open_facts}, $\V{\fg}$ having ACC implies that $\V{\fg}$ is not homomorphically open. So w.l.o.g.\ assume $\V{\fg}[q] \not \twoheadrightarrow \V{\fg}[r]$. Let $s$ be the stage at which both isomorphism types have stabilized. Take some $n^{[q]}$ that enters $W_{f(q)}$ only after stage $s$. Note that $n^{[q]}$ cannot be directly copied into $W_{f(r)}$. Assume $n^{[q]}$ goes through intermediate sets $W_{f(h_0)},..., W_{f(h_m)}$. This can only happen if there exist stages $s_0 < s_1 <...<s_{m+1}$, all greater than $s$, s.t.\ \[\V{\fg}[q] \cong_{s_0} \V{\fg}[h_0], \V{\fg}[h_0] \cong_{s_1} \V{\fg}[h_1],..., \V{\fg}[h_m] \cong_{s_{m+1}} \V{\fg}[r].\] But also note that obviously for all $i \in \{0,...,m\}$, $\V{\fg}[h_i]$ at stage $s_{i+1}$ is a homomorphic image of $\V{\fg}[h_i]$ at stage $s_i$. Therefore, $\V{\fg}[r]$ at stage $s_{m+1}$ is a homomorphic image of $\V{\fg}[q]$ at stage $s_0$. Also, $s < s_0,s_{m+1}$, so both have stabilized to the actual isomorphism type. This is a contradiction to our assumption of $\V{\fg}[q] \not \twoheadrightarrow \V{\fg}[r]$. So we know that $n^{[q]}$ also cannot be copied indirectly into $W_{f(r)}$, which means $W_{f(q)}$ and $W_{f(r)}$ are different.
\end{proof}

Note that $\eqce$ is a $\Pi^0_2$ equivalence relation, thus $\iV{\fg}$ must be $\Pi^0_2$. In the next theorem, we prove that $\iV{\fg}$ must be $\Delta^0_2$ for classes that have either a decidable word or isomorphism problem. However, being $\Delta^0_2$ does not already imply being below $\eqce$ (see Theorem \ref{no_universal_pi2}). We denote by $\WP(\V{})$ the word problem for finite presentations in a variety $\V{}$.

\begin{theorem}\label{acc_wp_ip_delta2}
Let $\V{}$ be a variety where $\V{\fg}$ has ACC. Then
\begin{enumerate}
    \item $\WP(\V{})$ is decidable $\implies$ $\iV{\fg}$ is $\Delta^0_2$,
	\item $\IP(\V{})$ is decidable $\implies$ $\iV{\fg}$ is $\Delta^0_2$.
\end{enumerate}
\end{theorem}

\begin{proof}

\begin{enumerate}
    \item 
    We only need to give a $\Sigma^0_2$-definition, as Theorem \ref{acc_upper_bound} implies membership in $\Pi^0_2$.
    
    Let $\A, \B \in \V{\fg}$ be given by presentations $\langle X_k \mid W_e\rangle$ and $\langle X_j \mid W_i\rangle$ respectively. Let $\approx_{W_e},\approx_{W_i}$ be the congruence relations generated by $W_e,W_i$ respectively.
    Existence of an isomorphism from $\A$ to $\B$ can be stated as follows: $\A \cong \B \iff \exists \varphi: X_k \to \T(X_j)$, s.t.\ its extension $\varphi^\prime: \T(X_k) \to \T(X_j)$ to all terms (as in the proof of Observation \ref{fp_sigma_1}) is an isomorphism from $\T(X_k)/\approx_{W_e}$ to $\T(X_j)/\approx_{W_i}$.
    
    Asking for the existence of $\varphi$ adds one existential quantifier, so it is sufficient to show that checking whether $\varphi^\prime$ is an isomorphism is a $\Sigma^0_2$ statement. We have to check well-definedness, injectivity, and surjectivity --- it then follows from the definition of $\varphi^\prime$ that it is a homomorphism. We use ACC to state these conditions in terms of the word problem for a finite presentation --- which is decidable by assumption. By $=_{\A,t}$ (and analogously $=_{\B,t}$) we denote the equivalence relation associated to the word problem for $\langle X_k \mid W_{e,t}\rangle$.
    \begin{itemize}
        \item Well-definedness on equivalence classes: $\exists$ stage $t_0$, s.t. $\forall t > t_0$ and $\forall$ terms $s_1,s_2$, if $s_1 =_{\A,t} s_2$ then $\varphi(s_1) =_{\B,t} \varphi(s_2)$.
        \item Injectivity: $\exists$ stage $t_0$, s.t. $\forall t > t_0$ and $\forall$ terms $s_1,s_2$, if $\varphi(s_1) =_{\B,t} \varphi(s_2)$ then $s_1 =_{\A,t} s_2$.
        \item Surjectivity: for every generator $x \in \{x_0,...,x_{j-1}\}$, $\exists$ term $v$ and $\exists$ stage $t$, s.t.\ $v$ is generated by $\varphi(x_0),...,\varphi(x_{k-1})$ and $v =_{\B,t} x$.
    \end{itemize}
    
    \item 
	By the Limit Lemma, a set is $\Delta^0_2$ if and only if its characteristic function is the limit of a computable function. Consider the function \[
	f(\langle k,e\rangle,\langle l,i \rangle, s) = \begin{cases}
		1, &\text{if }\langle X_k \mid  W_{e,s}\rangle_{\V{n}} \cong \langle X_l \mid  W_{i,s}\rangle_{\V{n}}\\
		0, &\text{otherwise}.
	\end{cases}
	\]
	By decidability of $\IP(\V{})$, $f$ is computable. By ACC, the isomorphism type of c.e.\ presentations in $\V{\fg}$ eventually stabilizes. Therefore, the characteristic function
	\[ \chi_{\iV{\fg}}(\langle k,e\rangle,\langle l,i \rangle) = \lim_{s \to \infty} f(\langle k,e\rangle,\langle l,i \rangle, s) \]
    of $\iV{n}$ is the limit of the computable function $f$.
		\end{enumerate}
\end{proof}

\begin{corollary}
    If $\V{\fg}$ not only has ACC but also either $\WP(\V{})$ or $\IP(\V{})$ is decidable, then $\V{\fg} <_c \eqce$. 
\end{corollary}

This follows from Theorem \ref{acc_upper_bound} and Theorem \ref{acc_wp_ip_delta2}, and in particular applies to $\CS{}$ (by Theorem \ref{isomorphism_problem_decidable}). With this, we are now ready to present a picture of the isomorphism relations we have discussed. Figure \ref{fig:below=ce} includes the results below $\eqce$ that we have shown so far and a classification of $\UF{2}{1}$, which will be shown in Section \ref{sec:above_eqce}.
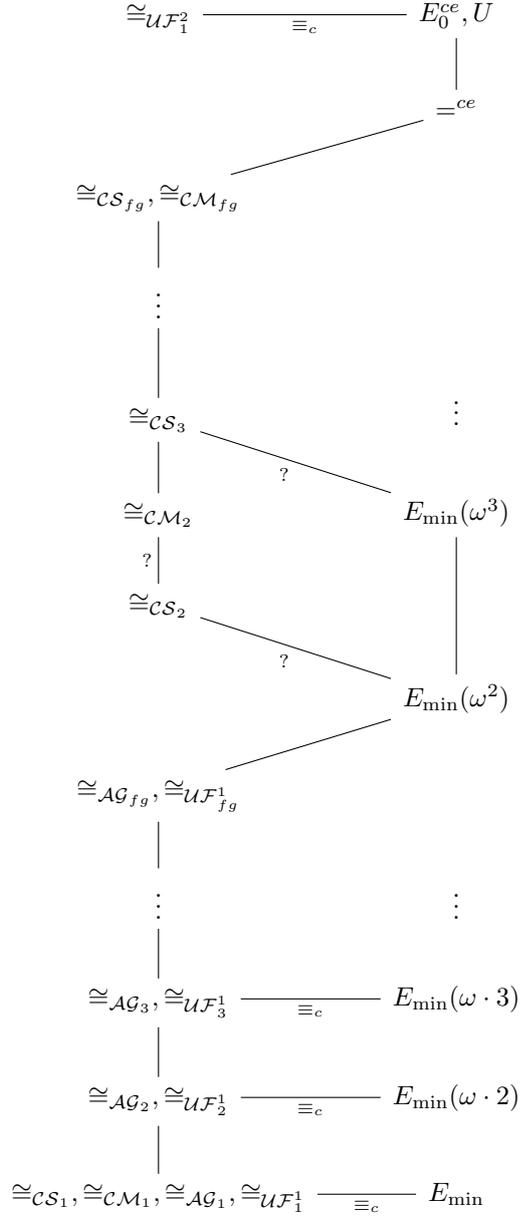
\begin{figure}
\begin{center}
    \begin{tikzcd}
        \iUF{2}{1}& \ezero,U \arrow[dash]{l}{\equiv_c}\\
        & \eqce \arrow[dash]{u}\\
        \iCS{fg}, \iCM{fg} \arrow[dash]{ur}\\
        \vdots\arrow[dash]{u}\\
        \iCS{3} \arrow[dash]{u} & \vdots\\
        \iCM{2} \arrow[dash]{u} & \emin(\omega^3) \arrow[dash]{ul}{?}\\
        \iCS{2} \arrow[dash]{u}{?}\\
        & \emin(\omega^2) \arrow[dash,uu] \arrow[dash,ul,"?"]\\
        \iAG{fg}, \iUF{1}{fg} \arrow[dash]{ur}\\
        \vdots \arrow[dash]{u} & \vdots\\
        \iAG{3}, \iUF{1}{3} \arrow[dash]{u} & \emin(\omega \cdot 3) \arrow[dash]{l}{\equiv_c}\\
        \iAG{2}, \iUF{1}{2} \arrow[dash]{u} & \emin(\omega \cdot 2) \arrow[dash]{l}{\equiv_c}\\
        \iCS{1}, \iCM{1}, \iAG{1}, \iUF{1}{1} \arrow[dash]{u} & \emin \arrow[dash]{l}{\equiv_c}
    \end{tikzcd}
    
\end{center}
\caption{The isomorphism relations investigated in this paper, compared to well-understood equivalence relations. A question mark indicates that we do not know whether or not the reducibility is strict. Apart from strict reducibilities, we omitted non-reducibility results.}
\label{fig:below=ce}
\end{figure}

In the remaining part of this section, we give two general results about another important benchmark relation, $\emax$. The first one states that its degree cannot contain isomorphism relations on finitely generated algebras with a decidable isomorphism problem. The second one states that $\emax$ is not reducible to isomorphism relations on classes with ACC and a decidable isomorphism problem.

\begin{proposition}\label{emax_does_not reduce_acc_ip}
    Let $\V{}$ be a variety where $\V{\fg}$ has ACC and $\IP(\V{})$ is decidable. Then $\emax \not\leq_c \iV{\fg}$.
\end{proposition}
\begin{proof}
Assume towards a contradiction that there exists such a reduction $f$ from $\emax$ to $\iV{\fg}$. Consider the c.e.\ sets $W_e$ and $W_i$, constructed as follows (with $e$,$i$ fixed by the Recursion Theorem): 
    \begin{itemize}
        \item We start with both $W_e$ and $W_i$ empty. 
        \item Stage $2k$. At this point, $W_e$ and $W_i$ are in the same $\emax$-class. We wait until we see that the algebras presented by the images under $f$ of $W_e$ and $W_i$ are isomorphic. We can decide this by our assumption that $\IP(\V{})$ is decidable. Once that happens, we put $k$ into $W_e$. If that never happens, $f$ is not a reduction.
        \item Stage $2k + 1$. At this point, $W_e$ and $W_i$ are in different $\emax$-classes. We wait until we see that the presentations are non-isomorphic (again, this is decidable by our assumption that $\IP(\V{})$ is decidable). Once that happens, put $k$ into $W_i$. If that never happens, $f$ is not a reduction.  
        \end{itemize}
    
    Now, for each $k$, the image of $W_e$ after stage $2(k+1)$ is a presentation with a congruence relation that is strictly bigger than the congruence relation of the image of $W_e$ after stage $2k$. This implies the existence of a strictly ascending sequence of congruence relations, a contradiction to ACC. 
\end{proof}

From this, we immediately also get that for all $n \geq 1$, $\emax \not\leq_c \iV{n}$.

\begin{proposition}\label{nothing_in_emax_deg}
        Let $\V{}$ be a variety of finite type and with $\IP(\V{})$ decidable. Then neither $\V{\fg}$ nor $\V{n}$, for $n \geq 1$, can be in the degree of $\emax$.
\end{proposition} 

\begin{proof}

    Assume there exists a variety $\V{}$ of finite type and with $\IP(\V{})$ decidable, s.t.\ $\iV{n}$, for some $n$, or $\iV{\fg}$ is $c$-equivalent to $\emax$. We do the case of $\iV{\fg}$, the other case follows similarly. First, notice that $\V{\fg}$ having ACC would contradict Proposition \ref{emax_does_not reduce_acc_ip}. So we can assume that there exists an algebra in $\V{\fg}$ whose congruence lattice has an infinite strictly ascending chain of congruences.
    Let $R$ be the congruence relation generated by the union of this sequence. Then $R$ cannot be finitely generated, otherwise one of the quotients would contain the generating equations. So in particular, $R$ cannot be the full congruence because, for varieties of finite type, the full congruence is finitely generated. Moreover, $R$ and the full congruence present non-isomorphic algebras. Now, consider the reduction $f$ from $\iV{\fg}$ to $\emax$. If we compose it with a function from $2^\omega$ to $2^\omega$ that fills the sets up below the maximal element, we get another reduction that is well-defined on c.e.\ sets. So w.l.o.g.\ assume $f$ is well-defined on c.e.\ sets. By the monotonicity lemma (Lemma \ref{monotonicity_lemma}), $f$ needs to map both $R$ and the full congruence to a superset of all of $W_{f(g(\emptyset))} \neq W_{f(g(\{0\}))} \neq W_{f(g(\{0,1\}))} \neq ...$. Therefore, both are mapped to an infinite set. This, however, means that they are mapped to $\emax$-equivalent sets, a contradiction.
\end{proof}

Furthermore, note that if an isomorphism relation is above $\emin$ (which is the case for all of our example relations), then it cannot be reducible to $\emax$. If it were, then $\emin$ would be reducible to $\emax$, a contradiction.

\section{Above \texorpdfstring{$\eqce$}{=ce} and \texorpdfstring{$\Sigma^0_3$}{Sigma 0 3}-completeness.}\label{sec:above_eqce}

This section collects our results above $\eqce$. We show that $\UF{2}{1}$ is $\Sigma^0_3$-complete. This is another example of a $\Sigma^0_3$-complete class, besides $6$-generated groups \cite{andrews2024two}. Then, we use homomorphic openness to derive the general results that no isomorphism relation on a homomorphically open class can be below $\eqce$.

We introduce a $\Sigma^0_3$-complete equivalence relation $U$, from \cite{andrews2025analogues}, which we then reduce to $\iUF{2}{1}$.

\begin{definition}\label{U_rel}
    For a c.e.\ set $W_e \subseteq \Z$ (via some computable encoding) and $x\in \Z$, we define $W_e+x = \{w+x\mid  w \in W_e\}$ to be the \emph{shift} of $W_e$ by $x$. 
    Define $U$ to be the equivalence relation on $\nat$ such that $eUi$ if and only if there is some $x\in \Z$ such that $W_e + x = W_i$.
\end{definition}

\begin{theorem}[\cite{andrews2025analogues}]
    $U$ is a $\Sigma^0_3$-complete equivalence relation.
\end{theorem}

\begin{theorem}\label{thm:uf21-complete}
	$\iUF{2}{1}$ is a $\Sigma^0_3$-complete equivalence relation.
\end{theorem}

\begin{proof}
	We show that $U$ computably reduces to $\iUF{2}{1}$ where $\UF21$ has the language $\{f,g\}$ where $f,g$ are two unary functions. 
    Consider the algebra (see Figure \ref{fig:uf21-complete}) defined by universe $\{a_i, b_i, b_i^\prime, b_i^{\prime\prime}, c_i, c_i^\prime, c_i^{\prime\prime}  \mid  i \in \Z\}$ and 
    \begin{itemize}
        \item $f(a_i) = b_i$ and $g(a_i) = c_i$,
        \item $f(b_i) = a_{i+1}$ and $g(c_i) = a_{i-1}$,
        \item $g(b_i) = b_i^\prime$ and $f(c_i) = c_i^\prime$,
        \item $f(b_i^\prime) = g(b_i^\prime) = b_i^{\prime\prime} = f(b_i^{\prime\prime}) = g(b_i^{\prime\prime})$,
        \item $f(c_i^\prime) = g(c_i^\prime) = c_i^{\prime\prime} = f(c_i^{\prime\prime}) = g(c_i^{\prime\prime})$.
    \end{itemize}

\begin{figure}
\begin{center}
    \begin{tikzcd} 
       & &                       &  b''_{i} \arrow[loop right,"{f,g}"]   &           & &          \\
       & &                       &  b'_{i} \arrow[u,"{f,g}"]   &           &  &         \\
       & b_{i-1}\arrow[rd,"f"] &                       &  b_{i} \arrow[rd,"f"]\arrow[u,"g"]   &           &   b_{i+1}\arrow[rd]&        \\
\cdots \arrow[ru]& & a_i\arrow[ru,"f"]\arrow[ld,"g"]         &                     & a_{i+1}\arrow[ld,"g"]\arrow[ru,"f"]  &  &\cdots\arrow[ld]   \\
       & c_{i}\arrow[lu]&                       &  c_{i+1}\arrow[lu,"g"]\arrow[d,"f"]            &           &  c_{i+2}\arrow[lu,"g"]  &               \\
       & &                       &  c'_{i+1}\arrow[d,"{f,g}"]&&&\\
       & &                       &  c''_{i+1}\arrow[loop right,"{f,g}"]&&&
\end{tikzcd}
\end{center}
\caption{The algebra in Theorem \ref{thm:uf21-complete}}
\label{fig:uf21-complete}
\end{figure}
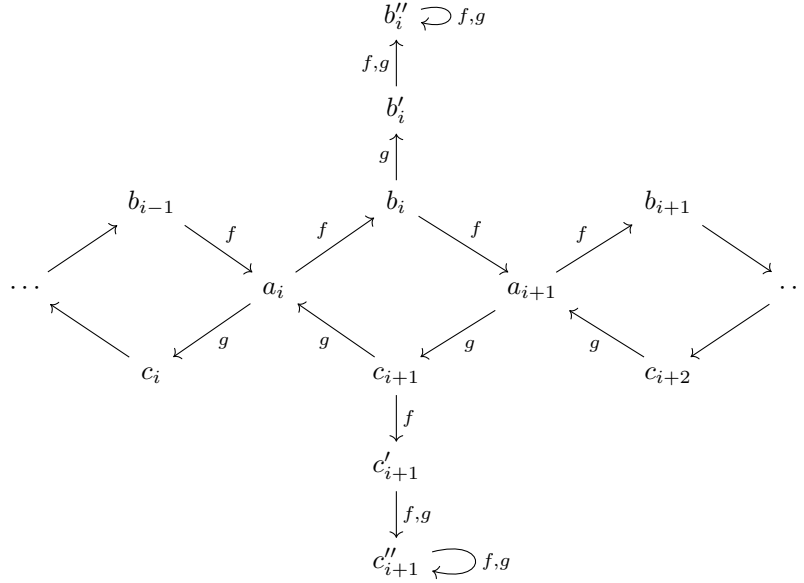

    Observe that this algebra is generated by $a = a_0$ and has a c.e.\ presentation in $\UF{2}{1}$.
	
	We now give a computable reduction $h$ from $U$ to $\UF{2}{1}$. Given some $e \in \nat$, the presentation $h(e)$ is given by enumerating the above presentation together with enumerating 
	\[
	b'_z = b''_z \text{ and } c_z^\prime = c_z^{\prime\prime}  
	\]
    into $h(e)$ when we see that $z$ enters $W_e$.
	
    It is thus clear that $h$ is computable. We need to verify that it is a reduction. Assume $eUi$, and that $W_e+z = W_i$. Then it is straightforward to check that the homomorphism defined by $a_n\mapsto a_{n+z}$ is an isomorphism of $h(e)\cong h(i)$. On the other hand, assume $h(e) \cong h(i)$. Then the isomorphism from $h(e)$ to $h(i)$ must send $a_0$ to some $a_z$, as they are the elements that are fixed by $g^2f^2$.  
    Then it is straightforward to check that $W_e + z = W_i$ and that $e U i$. 
\end{proof}

As a side note, this improves an auxiliary result in \cite{andrews2024two}, where $\Sigma^0_3$-completeness is shown for infinitely many unary function symbols. 

Not much is known about the word problem and isomorphism problem for algebras with unary function symbols. If the answer to the following question was yes, this would be an example where we have a decidable word or isomorphism problem for finite presentations, but a $\Sigma^0_3$-complete isomorphism relation on the c.e.\ presentations. It is not known whether such an example exists.
\begin{question}
    Does $\UF{m}{n}$, $m\geq 2, n\geq 1$, have a decidable word problem or isomorphism problem (i.e.\ for finite presentations)?
\end{question}

In Theorem \ref{acc_upper_bound}, we have shown that ACC implies reducibility to $\eqce$. It is open whether or not the reverse direction holds, i.e.\ not having ACC implying non-reducibility to $\eqce$. Towards this question, the strongest result that we have so far is the following.

\begin{proposition}\label{hom_open_implies_not_below_eqce}
    Let $\V{}$ be a variety where $\V{\fg}$ is homomorphically open. Then $\iV{\fg}\not\le \eqce$.
\end{proposition}

\begin{proof}
    Toward a contradiction, suppose $f$ is a reduction from $\iV{n}$ to $\eqce$. Let $\A\not\cong\B$ witness homomorphic openness.
    By Lemma \ref{effective_quotient_superset}, non-uniformly, pick codes $\langle g_i,i\rangle,\langle g_k,k\rangle$ for presentations of $\A$ and a code $\langle g_j,j\rangle$ for a presentation of $\B$, s.t.\ $W_i \subsetneq W_j \subsetneq W_k$. 
    Since $f$ is a reduction, $\V{\fg}[\langle g_i,i\rangle] \cong \V{\fg}[\langle g_k,k\rangle] \not\cong \V{\fg}[\langle g_j,j\rangle]$ implies $W_{f(\langle g_i,i\rangle)} = W_{f(\langle g_k,k\rangle)} \neq W_{f(\langle g_j,j\rangle)}$. 
    However, since $f$ is well-defined on c.e.\ sets, by Lemma \ref{monotonicity_lemma} it is monotone and thus we have $W_{f(\langle g_i,i\rangle)} \subsetneq W_{f(\langle g_j,j\rangle)} \subsetneq W_{f(\langle g_k,k\rangle)}$, a contradiction.
\end{proof}

\section{Further Questions}\label{sec:questions}

We conclude with a collection of open problems that naturally emerge from our analysis. These questions point toward a more refined understanding of the complexity landscape of isomorphism relations for c.e.\ algebras, particularly in relation to structural features such as the ascending chain condition (ACC) and benchmark degrees like $\eqce$ and $\Sigma^0_3$-completeness.

Our results so far suggest a dichotomy: all the concrete examples we have analyzed either satisfy ACC or exhibit maximal complexity, being $\Sigma^0_3$-complete. This motivates the following:

\begin{question}
Is there a class of c.e.\ algebras that does not satisfy ACC and whose isomorphism relation is not $\Sigma^0_3$-complete?
More broadly, are there isomorphism relations that lie strictly between $\eqce$ and the $\Sigma^0_3$-complete degree? Is there an isomorphism relation computably equivalent to $\eqce$?
\end{question}

Another striking feature of our examples is their comparability: all the isomorphism relations we studied are pairwise comparable under computable reducibility. However, we expect this to be an artifact of the current scope, and pose:

\begin{question}
Are there two classes of c.e.\ algebras whose isomorphism relations are incomparable under computable reducibility? More ambitiously, are there “natural” examples of such incomparability?
\end{question}

This connects to a broader structural question about the role of the degrees $\emin(\alpha)$ in classifying isomorphism relations. Note that $\emin(\alpha)$ has been defined only below $\omega^\omega$, with potential extensions up to $\varepsilon_0$, so this question admits some room for interpretation.

\begin{question}
Is every isomorphism relation for a class with ACC computably bi-reducible to some $\emin(\alpha)$, for a suitable ordinal $\alpha$?
\end{question}

Finally, in all non-$\Sigma^0_3$-complete examples we have considered, the complexity of the isomorphism relation increases strictly with the number of generators. This raises the following question, aiming to identify more subtle behavior:

\begin{question}
Is there a nontrivial variety $\mathcal{V}$ and an integer $n \geq 1$ such that $\iV{n} \equiv_c \iV{n+1}$,
but $\iV{n}$ is strictly below the $\Sigma^0_3$-complete degree?
\end{question}

\smallskip

Broadly speaking, these questions aim to move toward a structural characterization of the complexity of isomorphism relations for c.e.\ algebras. In particular, we seek: a structural account of when an isomorphism relation lies strictly below $\eqce$;
and a clear understanding of the structural features that force an isomorphism relation to be $\Sigma^0_3$-complete.

\bibliographystyle{alpha}
\bibliography{literatur}

\end{document}